\documentclass{article}
\usepackage[utf8]{inputenc}
\usepackage{geometry}

\usepackage{appendix}

\usepackage{fullpage}
\usepackage[inline]{enumitem}

\usepackage{amsmath,amsfonts,amssymb,amsthm,epsfig,epstopdf,url,array,mathrsfs}
\usepackage[colorlinks=true,citecolor=red,linkcolor=blue,pdfpagetransition=Blinds]{hyperref}

\usepackage{cleveref}

\numberwithin{equation}{section}

\usepackage[sort,nocompress]{cite}

\newcommand{\R}{\mathbb R}
\newcommand{\cE}{{\mathcal E}}
\newcommand{\N}{\mathbb{N}} 
\newcommand{\cN}{\mathcal{N}} 

\newcommand{\norme}[1]{\left\lVert#1\right\rVert}
\newcommand{\ov}[1]{\overline{#1}}
\newcommand{\ds}{2_s^\star}
\newcommand{\dsa}[1]{2_{#1}^\star}
\newcommand{\pv}{\textnormal{p.v.}}
\newcommand{\weakly}{\rightharpoonup}
\newcommand{\eps}{\varepsilon}

\theoremstyle{plain} 
\newtheorem{prop}{Proposition}[section] 
\newtheorem{theo}[prop]{Theorem}

\newtheorem{lem}[prop]{Lemma}

\theoremstyle{definition}

\def\dx{\,\textnormal{d}x}

 \def\d{\,\textnormal{d}}
\def\esp{{\mathbb{E}}}

\newcommand{\vertiii}[1]{{\left\vert\kern-0.25ex\left\vert\kern-0.25ex\left\vert #1 
    \right\vert\kern-0.25ex\right\vert\kern-0.25ex\right\vert}}

\usepackage{color}
\usepackage[dvipsnames]{xcolor}

\begin{document}

\title{Existence and convergence of solutions to fractional pure critical exponent problems}
\author{V\'ictor Hern\'andez-Santamar\'ia\thanks{Instituto de Matem\'aticas, Universidad Nacional Aut\'onoma de M\'exico, Circuito Exterior, C.U., C.P. 04510 CDMX, Mexico. E-mails: \texttt{victor.santamaria@im.unam.mx}, \ \texttt{alberto.saldana@im.unam.mx}} \and Alberto Salda\~{n}a\footnotemark[1]}
\date{}

\maketitle
 \begin{abstract}
We study existence and convergence properties of least-energy symmetric solutions (l.e.s.s.) to the pure critical exponent problem
\begin{equation*}
(-\Delta)^su_s=|u_s|^{\ds-2}u_s, \quad u_s\in D^s_0(\Omega),\quad 2^\star_s:=\frac{2N}{N-2s},
\end{equation*}
where $s$ is any positive number, $\Omega$ is either $\R^N$ or a smooth symmetric bounded domain, and $D^s_0(\Omega)$ is the homogeneous Sobolev space. Depending on the kind of symmetry considered, solutions can be sign changing. We show that, up to a subsequence, a l.e.s.s. $u_s$ converges to a l.e.s.s. $u_{t}$ as $s$ goes to any $t>0$. In bounded domains, this convergence can be characterized in terms of an homogeneous fractional norm of order $t-\varepsilon$.  A similar characterization is no longer possible in unbounded domains due to scaling invariance and an incompatibility with the functional spaces; to circumvent these difficulties, we use a suitable rescaling and characterize the convergence via cut-off functions. If $t$ is an integer, these results describe in a precise way the nonlocal-to-local transition. Finally, we also include a nonexistence result of nontrivial nonnegative solutions in a ball for any $s>1$.
\medbreak

(Primary); 
35R11, 
35J35, 
35J40. 

\end{abstract}

\section{Introduction}

In this paper we study existence and convergence properties of solutions to pure critical problems such as
\begin{equation}\label{eq:frac_crit_exp}
(-\Delta)^s u_s=|u_s|^{\ds-2}u_s, \qquad 
u_s\in D_0^s(\Omega),\qquad 2^\star_s:=\frac{2N}{N-2s},
\end{equation}
where $N\geq 1$, $s>0$, $N>2s$, $(-\Delta)^s$ is the (possibly higher-order) fractional Laplacian, $\Omega$ is either $\R^N$ or a smooth bounded domain of $\R^N$, and $D^s_0(\Omega)$ is the homogeneous (fractional) Sobolev space, namely, the closure of $C^\infty_c(\Omega)$ with respect to the Gagliardo seminorm $\|\cdot\|_s$, given by \eqref{def:norm_s} below.  See Section~\ref{sec:prelim} for precise definitions and main properties of the operator $(-\Delta)^s$ and the space $D^s_0(\Omega)$.

Problem \eqref{eq:frac_crit_exp} is an important paradigm in nonlinear analysis of PDEs and plays an important role in the study of the well-known Yamabe problem in differential geometry and its generalizations. Moreover, the fractional Laplacian plays an important role in the study of anomalous and nonlocal diffusion, which appears for instance in continuum mechanics, graph theory, and ecology, see \cite{BV16} and the references therein. 

For $s=1$ there is an extensive literature on existence of solutions of \eqref{eq:frac_crit_exp} using different methods, see, for instance, \cite{GNN79,Ding86,HV94,dPMP11,dPMP13,Clapp16,FP,CSS} and the references therein. When $s\in \mathbb N$ with $s\geq 2$, equation \eqref{eq:frac_crit_exp} is the pure critical exponent problem for the polyharmonic operator and existence of solutions has been studied in~\cite{EFJ90,Gru95,BWW03,GGS03,BSW04,Ge05}.  In the fractional setting, existence results in $\R^N$ for $s\in(0,1)$ are available in~\cite{Lie83,DdPS13,GM16,fang,ABR19,T20}, and for $s>1$ it is known that \eqref{eq:frac_crit_exp} has a family of radially symmetric solutions in $\R^N$, see~\cite{chen}.  

\medskip

The first objective of this paper is to present a unified approach to show existence of solutions of \eqref{eq:frac_crit_exp} for any $s\in(0,\infty)$. For $\Omega$ bounded this is a problem that depends strongly on the geometry of the domain, whereas for $\Omega=\R^N$ all positive solutions of \eqref{eq:frac_crit_exp} are completely characterized and therefore we are particularly interested in nonradial sign-changing entire solutions. The second objective, is to investigate the convergence properties of solutions, namely, if $(u_{s_k})_{k\in\N}$ are solutions of \eqref{eq:frac_crit_exp} (with $s_k\in(0,\infty)$ instead of $s$), then what can be said about the limit of $u_{s_k}$ as $s_k\to s_0>0$. For the critical nonlinearity $f(u)=|u|^{2^\star_s-2}u$, we are not aware of any previous result in this direction. 
  
Although problem \eqref{eq:frac_crit_exp} has a variational structure (with energy functional given by \eqref{eq:func_RN}), variational methods face several compactness issues, mainly due to the following scaling invariance
\begin{align}\label{issues}
\|u\|_s=\|u_{\lambda,\xi}\|_s,\qquad 
\int_{\R^N}|u|^{2^\star_s}=\int_{\R^N}|u_{\lambda,\xi}|^{2^\star_s},\qquad 
u_{\lambda,\xi}(x):=\lambda^{\frac{N}{2}-s}u(\lambda x+\xi),
\end{align}
for $u\in D_0^s(\R^N)$, $\lambda>0$, and $\xi\in\R^N$.

One way to overcome this difficulty, is to search for solutions within a symmetric framework. In this way, we regain some compactness to achieve least-energy solutions (among symmetric functions) and we also obtain directly important information about the shape of solutions, which can be used to guarantee multiplicity results.

Following the framework from~\cite{Clapp16,CLR18,CS20}, let us introduce some notation.  Let $G$ be a closed subgroup of the group $O(N)$ of linear isometries of $\R^N$ such that
\begin{itemize}
\item[{\bf ($A_1$)}] for each $x\in\mathbb R^N$, either $\dim(Gx)>0$ or $Gx=\{x\}$,
\end{itemize}
where $Gx:=\{gx\::\: g\in G\}$ is the $G$-orbit of $x$.
Let $\phi:G\to {\mathbb Z}_2:=\{-1,1\}$ be a continuous homomorphism of groups (\emph{i.e.} $\phi(g\circ h)=\phi(g)\phi(h)$) and let $\Omega$ be a $G$-invariant set of $\mathbb R^N$ (\emph{i.e.}, $Gx\subset \Omega$ if $x\in \Omega$). A function $u:\Omega\to \R$ is said to be \emph{$\phi$-equivariant} if
\begin{equation}\label{def:equi}
u(gx)=\phi(g)u(x) \quad\textnormal{for all } g\in G\text{ and } x\in \Omega.
\end{equation}

Depending on $\phi,$ it could happen that \eqref{def:equi} is only satisfied by $u\equiv 0$, for instance, if $G=O(N)$ and $\phi(g)$ is the determinant of $g\in G$. To avoid this, we need to impose some condition on $\phi$, namely, that
\begin{itemize}
\item[{\bf ($A_2$)}] there exists $\xi\in\mathbb R^N$ such that $\{g\in G:g\xi=\xi\}\subset \ker \phi:=\{g\in G: \phi(g)=1\}$.  
\end{itemize}
Under this condition, the space 
\begin{equation}\label{D0phi:def}
D_0^s(\Omega)^{\phi}:=\left\{u\in D_0^{s}(\Omega):u \textnormal{ is $\phi$-equivariant}\right\}
\end{equation}
has infinite dimension, see~\cite[Theorem 3.1]{BCM05}.

Our first result concerns bounded domains.  Let $\Omega^G:=\{x\in\Omega\::\: gx=x\text{ for all }g\in G\}$ be the set of $G$-fixed points of $\Omega$  and let ${\mathbb N}_0:=\mathbb N\cup \{0\}$.
\begin{theo}\label{main:thm:bdd} Assume that $G$ and $\phi$ verify assumptions \textnormal{\textbf{($A_1$)}} and \textnormal{\textbf{($A_2$)}}. Let $N\geq 1$ and let $\Omega\subset \R^N$ be a smooth bounded $G$-invariant domain such that $\Omega^G=\emptyset$.
 \begin{enumerate}
  \item (Existence) For every $s>0$ with $N>2s$ there is a $\phi$-equivariant least-energy solution $u_s$ of
  \begin{equation}
(-\Delta)^s u_s=|u_s|^{\ds-2}u_s, \qquad 
u_s\in D_0^s(\Omega)^\phi\backslash \{0\}.\label{Ps}
\end{equation}
The solution is sign-changing if $\phi:G\to\{-1,1\}$ is surjective.
\item (Convergence) Let $(s_k)_{k\in\N}\subset(0,\frac{N}{2})$ such that $s_k\to s$ as $k\to\infty$ with $N>2s>0$, and let $u_{s_k}$ be a $\phi$-equivariant least-energy solution of
\begin{equation*}
(-\Delta)^{s_k} u_{s_k}=|u_{s_k}|^{2^\star_{s_k}-2}u_{s_k}, \qquad 
u_{s_k}\in D_0^{s_k}(\Omega)^\phi\backslash \{0\}.
\end{equation*}
Then, up to a subsequence, there is a $\phi$-equivariant least-energy solution $u_s$ of \eqref{Ps} such that
\begin{align}\label{conv:bdd}
 u_{s_k}\to u_s\quad \text{ in }D^{t}_0(\Omega) \text{ as $k\to\infty$ for all }t\in[0,s).
\end{align}
\end{enumerate}
\end{theo}
Different choices for $G$ and $\phi$ in Theorem~\ref{main:thm:bdd} produce different kinds of solutions.  For instance, if $G=O(N)$ and $\phi\equiv 1$, then Theorem~\ref{main:thm:bdd} yields a solution $u_s$ of \eqref{Ps} which is radially symmetric.  On the other hand, if $G=G_i$ and $\phi=\phi_i$ are given as in \eqref{G:def} and \eqref{phi:def} below, then Theorem~\ref{main:thm:bdd} guarantees that problem \eqref{Ps} has at least $\lfloor \frac{N}{4} \rfloor$ nonradial sign-changing solutions, where $\lfloor x \rfloor$ denotes the greatest integer less than or equal to $x$; this existence result is new even in the local case for $s\in \mathbb N$ with $s\geq 3$ (the cases $s=1$ and $s=2$ are shown in~\cite{Clapp16,CS20} respectively).

To prove the first part of Theorem~\ref{main:thm:bdd} (existence), we extend to the fractional setting the strategy used in~\cite{Clapp16,CLR18,CS20} for local problems, where a symmetric-concentration compactness argument is used.  The main difficulty in this extension is the adaptation of a Brezis-Kato-type argument which is based on direct calculations for the Laplacian.  Direct computations are much harder in nonlocal problems (specially in the higher-order regime $s>1$). We overcome this difficulty using interpolation inequalities and sharp Hardy-Littlewood-Sobolev inequalities among other tools, see Section~\ref{sec:bk}.

The proof of the convergence result relies strongly on the $\phi$-equivariance of the solutions, which yields the necessary compactness to extract a convergent subsequence and to guarantee that the limit is a least-energy $\phi$-equivariant solution as well.  We remark that \eqref{conv:bdd} also holds in the standard Sobolev norm $\|\cdot\|_{H^{s-\delta}}$, which is equivalent to the homogeneous norm $\|\cdot\|_{s-\delta}$ in $D^{s-\delta}_0(\Omega)$ with $\Omega$ bounded.  After Theorem~\ref{main:thm:unbdd} we comment more on these results and compare our findings with previously known convergence results for subcritical problems.

The assumption $\Omega^G=\emptyset$ is fundamental, since the existence of solutions of critical problems is closely related to the geometry of the domain. Indeed, a consequence of the Pohozaev identity is that, if $\Omega$ is star-shaped and $s=1$, then \eqref{Ps} only admits trivial solutions. Although for any $s\in(0,\infty)$ there  are versions of the Pohozaev identity (see~\cite[Corollary 1.7]{rs15}), a general nonexistence result as in the case $s=1$ is, as far as we know, not available for \eqref{Ps} if $s\neq 1$. This is because the nonexistence proof also requires a unique continuation principle and the existence of a suitable extension of the solution to $\R^N$.

Using maximum principles, one can show nonexistence of \emph{nonnegative} solutions in starshaped domains for \eqref{Ps} if $s\in(0,1)\cup\{2\}$, see \cite[Corollary 1.3]{RS14} and \cite[Theorem 7.33]{GGS03}. On balls, the nonexistence of nonnegative solutions is also known for $s\in\N$, see \cite{LS08} (see also \cite[Theorem 7.34]{GGS03}). Using the Pohozaev identities from \cite{RS14} and a fractional higher-order Hopf Lemma from \cite{AJS18_poisson}, we can extend this nonexistence result to any $s>1.$ 

\begin{prop}\label{prop:nonex}
 Let $\alpha\in (0,1)$, $s>1$, $N>2s$, and let $B:=\{x\in\R^N\::\: |x|<1\}$. The problem 
\begin{equation}\label{nonex:eq}
(-\Delta)^su=|u|^{\ds-2}u, \qquad 
u\in D_0^s(B)\cap C^\alpha(\overline{B}),
\end{equation}
does not admit nontrivial nonnegative solutions.
\end{prop}
Maximum principles (and in particular Hopf Lemmas) do not hold in general domains if $s>1$; for instance, positivity preserving properties fail in ellipses  for $s\in(1,\frac{3}{2}+\sqrt{3})$, see \cite{AJS18_ellipse}, in dumbbell domains for $s\in (m,m+1)$ with $m$ odd, see \cite[Theorem 1.11]{AJS18_poisson}, and in two disjoint balls for $s\in (m,m+1)$ with $m$ odd, see \cite[Theorem 1.1]{AJS18_loss} (curiously, this last set has a \emph{positive} Green's function if $s\in(m,m+1)$ and $m$ is \emph{even}, see \cite[Theorem 1.10]{AJS18_poisson}).

\medskip

Next we present our existence and convergence results for \emph{entire solutions}, namely, when $\Omega=\R^N$.  This setting is more delicate for several reasons.  For the existence part, there is an inherent lack of compactness due to the scaling and translation invariance \eqref{issues}.  This is controlled in our proofs using the symmetric structure of $D_0^s(\R^N)^\phi$.  On the other hand, the characterization of the convergence of solutions faces a problem regarding the incompatibility of the functional spaces.  To be more precise, by the Sobolev inequality,
\begin{align*}
 D_0^s(\R^N)=D^s(\R^N):=\{u\in L^{2^\star_s}(\R^N)\::\: \|u\|_s<\infty\},
\end{align*}
(see Theorem \ref{thm:sobolev} below, see also \cite{BGV18} for a survey on homogeneous Sobolev spaces).  In particular, it is not true that $D^t(\R^N)\subset D^s(\R^N)$ for $t>s$, as it happens in bounded domains, and therefore it is not trivial to find a suitable norm to describe the convergence properties of solutions; for instance, a characterization such as \eqref{conv:bdd} is not possible in $\R^N$ since $u_s$ might not belong to $D^{t}(\R^N)$ for $t\neq s$. This is not a problem of local smoothness, but rather an incompatibility with the decay at infinity. In the following result we show that entire solutions converge when multiplied by an arbitrary function in $C^\infty_c(\R^N)$.

\begin{theo}\label{main:thm:unbdd} Assume that $N\geq 1$ and that $G$ and $\phi$ verify assumptions \textnormal{\textbf{($A_1$)}} and \textnormal{\textbf{($A_2$)}}. 
 \begin{enumerate}
  \item (Existence) For every $s>0$ with $N>2s$ there is a $\phi$-equivariant least-energy solution $w_s$ of
  \begin{equation}
(-\Delta)^s w_s=|w_s|^{\ds-2}w_s, \qquad 
w_s\in D^s(\R^N)^\phi\backslash \{0\}.\label{Us}
\end{equation}
The solution is sign-changing if $\phi:G\to\{-1,1\}$ is surjective.
\item (Convergence) Let $(s_k)_{k\in\N}\in(0,\frac{N}{2})$ such that $s_k\to s$ as $k\to\infty$ with $N>2s>0$, and let $w_{s_k}$ be a $\phi$-equivariant least-energy solution of
  \begin{equation*}
(-\Delta)^{s_k} w_{s_k}=|w_{s_k}|^{2^\star_{s_k}-2}w_{s_k}, \qquad 
w_{s_k}\in D^{s_k}(\R^N)^\phi\backslash \{0\}.
\end{equation*}
Then, up to a rescaled subsequence of $w_{s_k}$ denoted by $\widetilde w_{s_k}$, there is a $\phi$-equivariant least-energy solution $w_s$ of \eqref{Us} such that
\begin{align}\label{conv:unbdd}
 \eta\widetilde w_{s_k}\to \eta w_s\quad \text{ in }D^{t}(\R^N) \text{ as $k\to\infty$ for all }\eta\in C^\infty_c(\R^N) \text{ and }t\in[0,s).
\end{align}
In particular, $\widetilde w_{s_k}\to w_s$ in $L_{loc}^q(\R^N)$ as $k\to\infty$ for all $q\in[1,2^\star_s)$.
\end{enumerate}
\end{theo}
As in the bounded domain case, if $G=O(N)$ and $\phi\equiv 1$, then a solution $w_s$ of \eqref{Us} is a radially symmetric function, see~\cite{chen} for a study of this type of solutions.  If $G=G_i$ and $\phi=\phi_i$ are those given in \eqref{G:def} and \eqref{phi:def}, then Theorem~\ref{main:thm:unbdd} yields the existence of at least $\lfloor \frac{N}{4} \rfloor$ non-radial sign-changing solutions to \eqref{Us}. For $s\in(0,1)$, this existence result was proved in the recent paper~\cite{T20}, for $s=1$ it is shown in~\cite{Clapp16}, and for $s=2$ it is a particular case of~\cite[Theorem 1.1]{CS20}.  All these papers follow a strategy based on a symmetric-concentration compactness argument, but at a technical level they have important differences and none of them can be easily extended to guarantee existence of solutions in the whole higher-order range $s\in (1,\infty)$. In this sense, the method we present here is more flexible and universal. We emphasize that the solutions given by Theorem \ref{main:thm:unbdd} are different from those obtained in \cite{Ding86} for $s=1$, in \cite{BSW04} for $s\in\N$, and in \cite{fang,ABR19} for $s\in(0,1)$.

These entire solutions are obtained by a suitable rescaling of a concentrating energy-minimizing sequence, see Theorem~\ref{thm:concentration}, where it is also shown that the concentration point is necessarily a $G$-fixed point.  See also~\cite[Theorem 2.5]{Clapp16} for other variants of these type of results for the Laplacian.

In the convergence part in Theorem~\ref{main:thm:unbdd}, the rescaling $\widetilde w_{s_k}$ of the sequence $w_{s_k}$ is needed to avoid the scaling invariance \eqref{issues} typical in critical problems.  Without this rescaling it can happen that $w_{s_k}$ converges to 0 or that it diverges at every point.  A particularly useful rescaling is presented in Theorem~\ref{thm:conv:entire} via the condition \eqref{rscl}, which is convenient for technical reasons. The use of cut-off functions to characterize the convergence \eqref{conv:unbdd} is one of the main methodological contributions of this work and it requires delicate uniform estimates.

As far as we know, Theorems~\ref{main:thm:bdd} and~\ref{main:thm:unbdd} are the first results to consider the convergence of solutions in the critical regime (r.h.s. $|u|^{2^\star_s-2}u$) and for higher-order problems ($s\in(1,\infty)$). Previous convergence results were only available for subcritical problems (where the compactness of the embedding $D^s_0(\Omega)\hookrightarrow L^{p}(\Omega)$, $0<p<2^\star_s$, is the main tool) and only for $s_k\nearrow 1$, see~\cite{BS19,BS20,FBS20}.  For linear problems, the continuity of the solution map $s\mapsto v_s$ is considered in~\cite{BHS18} as $s\nearrow 1$ and the continuity and differentiability of this map at any $s\in(0,1)$ is studied in~\cite{JSW20}.

Furthermore, we mention that our convergence characterizations \eqref{conv:bdd} and \eqref{conv:unbdd} are stronger than those of~\cite{BS19,BS20,FBS20}, which are in terms of $L^2$ and $L^2_{loc}$ norms. Note that solutions of nonlinear equations with a potential (such as those considered in~\cite{BS19,BS20,FBS20}) would have $L^2$ as a common space for all solutions regardless of the value of $s$, but this is not the case for the pure critical exponent problem \eqref{Us}.

\medskip 

The paper is organized as follows. In Section~\ref{sec:prelim} we detail our symmetry setting and functional framework and exhibit a family of symmetry groups $G_i$ and surjective homomorphism $\phi_i$ that, together with Theorems~\ref{main:thm:bdd} and~\ref{main:thm:unbdd}, yield nonradial sign changing solutions. Section~\ref{sec:bk} contains the main technical tools used to show our main existence and convergence results. In Section~\ref{sec:c} we show a concentration result using a symmetric-concentration compactness argument.  Sections~\ref{sec:bdd} is devoted to the proof of Theorem~\ref{main:thm:bdd} and the nonexistence result stated in Proposition \ref{prop:nonex}. Finally, in Section \ref{sec:ubd}, we provide the proof of Theorem~\ref{main:thm:bdd}.

\section{Preliminaries}\label{sec:prelim}

In this section, we introduce the symmetric setting and functional framework to study the pure critical exponent problem \eqref{eq:frac_crit_exp}. We also detail the definition and some properties of the (possibly higher-order) fractional Laplacian and the homogeneous (fractional) Sobolev space.  

\subsection{Functional framework}

For $u\in C_c^\infty(\R^N)$ the fractional Laplacian of order $2\sigma$ is given by
\begin{equation*}
(-\Delta)^\sigma u(x)
=c_{N,\sigma}\pv\int_{\R^N}\frac{u(x)-u(y)}{|x-y|^{N+2\sigma}}\d{y}
=c_{N,\sigma}\lim_{\eps\to0}\int_{\{|x-y|>\eps\}}\frac{u(x)-u(y)}{|x-y|^{N+2\sigma}}\d{y}
\qquad \text{ for } x\in\R^N,
\end{equation*}
where $\pv$ means in the principal value sense,
\begin{align}\label{cnsigma}
c_{N,\sigma}:=4^\sigma\pi^{-N/2}\sigma(1-\sigma)\frac{\Gamma(N/2+\sigma)}{\Gamma(2-\sigma)} 
\end{align}
is a normalization constant, and $\Gamma$ is the usual gamma function.  Let $s=m+\sigma>1$ with $m\in\mathbb N$ and $\sigma\in(0,1)$. The fractional Laplacian of order $2s$ is given by
\begin{equation*}
(-\Delta)^su(x):=
\begin{cases}
\displaystyle
(-\Delta)^{\frac{m}{2}}(-\Delta)^\sigma(-\Delta)^{\frac{m}{2}}u(x), & \text{for $m$ even}, \\
\sum_{i=1}^{N}(-\Delta)^{\frac{m-1}{2}}\left(\partial_i (-\Delta)^\sigma\left(\partial_i(-\Delta)^{\frac{m-1}{2}}u(x)\right)\right), &\text{for $m$ odd}.
\end{cases}
\end{equation*}
We remark that other pointwise evaluations of $(-\Delta)^s$ are possible, see for example~\cite{AJS18,ssurvey}, and we refer to~\cite{AJS18_halfspace,AJS18_green,AJS18_poisson} for recent studies on boundary value problems involving higher-order fractional Laplacians.

For $s>0$ let $H^s(\R^N):=\left\{u\in L^2(\R^N):(1+|\xi|^2)^{\frac{s}{2}}\widehat u\in L^2(\R^N)\right\}$ denote the usual fractional Sobolev space, where $\widehat{u}$ denotes the Fourier transform of $u$. For $\Omega\subset \R^N$ a smooth open set, let $D^s_0(\Omega)$ be the closure of $C^\infty_c(\Omega)$ with respect to the norm 
\begin{equation}\label{def:norm_s}
\|u\|_s:=\left(\cE_s(u,u)\right)^{1/2},
\end{equation}
where 
\begin{equation}\label{eq:equiv_fourier}
\cE_{s}(u,v)=\int_{\R^N}|\xi|^{2s}\widehat u(\xi)\widehat v(\xi)\d{\xi}
\end{equation}%
is the associated scalar product.  If $\Omega=\R^N$, then we simply write $D^s(\R^N)$ instead of $D^s_0(\R^N)$.  Let $\mathcal H^s_0(\Omega):=\{u\in H^s(\R^N):u=0 \text{ on } \R^N\setminus \Omega\}$ equipped with the standard $H^s$-norm.  If $\Omega$ is bounded, then
\begin{equation}\label{eq:norm_equiv_h0s}
\ell_{1,s}\norme{u}_{\mathcal H_0^s(\Omega)}\leq \norme{u}_s\leq \norme{u}_{\mathcal H_0^s(\Omega)},
\end{equation}
where $\ell_{1,s}=2^{-1}\min\{1,\lambda_{s,1}\}$, $\lambda_{1,s}=\lambda_{1,s}(\Omega)$ is the first eigenvalue of $((-\Delta)^s,{\mathcal H}_0^s(\Omega))$, see, for example,~\cite{AJS18_loss}.

If $m\in \mathbb N$, $\sigma\in(0,1)$, $s=m+\sigma$, and $u,v\in D^s(\R^N)$,  then the following are equivalent expressions for $\cE_s$.
\begin{align}
 \cE_{\sigma}(u,v)&=
\frac{c_{N,\sigma}}{2}\int_{\R^N}\int_{\R^N}\frac{(u(x)-u(y))(v(x)-v(y))}{|x-y|^{N+2\sigma}}\d{x}\d{y},\nonumber\\
\cE_{s}(u,v)&=\begin{cases}
\cE_{\sigma}((-\Delta)^{\frac{m}{2}}  u,(-\Delta)^{\frac{m}{2}} v), & \text{if $m$ is even,}\\
\sum_{k=1}^{N}\cE_{\sigma}(\partial_k (-\Delta)^{\frac{m-1}{2}} u,\partial_k (-\Delta)^{\frac{m-1}{2}} v), & \text{if $m$ is odd.}
\end{cases} \label{bilin:def}
\end{align}
For some results we also consider $s\in\N$, in these cases we have that 
\begin{align*}
\cE_{s}(u,v)&=\begin{cases}
\displaystyle \int_{\R^N}(-\Delta)^{\frac{m}{2}}  u(-\Delta)^{\frac{m}{2}} v, & \text{if $m$ is even,}  \vspace{0.1 cm} \\
\displaystyle \int_{\R^N}\nabla (-\Delta)^{\frac{m-1}{2}} u\nabla (-\Delta)^{\frac{m-1}{2}} v, & \text{if $m$ is odd.}
\end{cases}
\end{align*}
In any case, for $s>0$, $\int_{\R^N}(-\Delta)^su(x)v(x)\dx=\cE_s(u,v)$ for $u\in C_c^\infty(\R^N)$ and $v\in D^s(\R^N)$, see, for example,~\cite{AJS18,AJS18_loss}.  Throughout the paper the $L^q$-norm is denoted by
\begin{align*}
 |f|_q:=\left(\int_{\Omega}|f(x)|^q\dx\right)^{1/q}\qquad \text{ for }q\in[1,\infty).
\end{align*}

We close this section with two important results.

\begin{theo}[Fractional Sobolev inequality]\label{thm:sobolev} Let $N\geq 1$, $s>0$, and $N>2s$. There is $\kappa_{N,s}>0$ such that $|u|_{\ds}\leq \kappa_{N,s}\|u\|_{s}$ for all $u\in D^s(\R^N),$ where
\begin{equation}\label{eq:best_constant}
\kappa_{N,s}=2^{-2s}\pi^{-s}\frac{\Gamma(\frac{N-2s}{2})}{\Gamma(\frac{N+2s}{2})} \left(\frac{\Gamma(N)}{\Gamma(\frac{N}{2})}\right)^{\frac{2s}{N}}.
\end{equation}
\end{theo}
\begin{proof}
 See~\cite[Theorem 1.1]{CT04}.
\end{proof}

\begin{theo}\label{thm:rellich_type}Let $\Omega$ be a bounded smooth domain, $s>0$, $N>2s$, $p\in[1,\ds)$, and $\eps\in(0,s]$. Then the embeddings $D^s_0(\Omega)\hookrightarrow D^{s-\eps}_0(\Omega)$ and 
$D_0^s(\Omega)\hookrightarrow L^p(\Omega)$ are compact.
\end{theo}
\begin{proof}
The compactness of the embedding $\mathcal H_0^s(\Omega)\hookrightarrow \mathcal H_0^{s-\eps}(\Omega)$ follows by interpolation theory, see~\cite{T78} (the space $\mathcal H_0^t(\Omega)$ is defined in~\cite[Sec.\,4.3.2 (1a)]{T78};
that $A=\mathcal H_0^s(\Omega)$ is an interpolation space between $A_0=H^{\lceil s \rceil}_0(\Omega)$ and $A_1=L^2(\Omega)$ is a consequence of \cite[Sec.\,4.3.2 Thm 2]{T78} together with~\cite[Sec.\,2.4.2 (10)]{T78}; finally, the compactness of the embedding $\mathcal H_0^s(\Omega) \hookrightarrow \mathcal H_0^{s-\eps}(\Omega)$ follows from 
\cite[Sec.\,1.16.4 Thm. 2 (a)]{T78} together with the compactness of the embedding $A_0\hookrightarrow A_1$, see~\cite[Sec.\,7.10]{GT01}). Then the compactness of $D^s_0(\Omega)\hookrightarrow D^{s-\eps}_0(\Omega)$ holds by the equivalence of norms \eqref{eq:norm_equiv_h0s}.  The embedding $D_0^s(\Omega)\hookrightarrow L^p(\Omega)$ is compact for $p\in[1,\ds)$ by~\cite[Theorem 1.5]{CT04}.
\end{proof}

\subsection{Symmetric setting}

Following~\cite{Clapp16,CLR18,CS20} we now present a series of results connecting the symmetric framework presented in the introduction and the variational structure of equation \eqref{eq:frac_crit_exp}. 

Let $G$ be a closed subgroup of $O(N)$ and let $\phi: G\to \mathbb{Z}_2:=\{-1,1\}$ be a continuous homomorphism of groups satisfying the properties {\bf ($A_1$)} and {\bf ($A_2$)} presented in the introduction. Let $\Omega$ be a $G$-invariant bounded smooth domain of $\mathbb R^N$ and recall the definition of $\phi$-equivariance given in \eqref{def:equi} and of the space $D_0^s(\Omega)^{\phi}$ given in \eqref{D0phi:def}. We say that $u\in D_0^s(\Omega)$ is a solution of 
\begin{equation}\label{eq:frac_domain}
(-\Delta)^su=|u|^{\ds-2}u, \qquad 
u\in D_0^s(\Omega),
\end{equation}
if $u$ is a critical point of the $C^1$-functional $J_s:D_0^s(\Omega)\to \R$ defined by
\begin{equation}\label{eq:func_RN}
J_s(u):=\frac{1}{2}\|u\|_s^2-\frac{1}{\ds}|u|_{\ds}^{\ds}.
\end{equation}
The next lemma is a type of principle of symmetric criticality in the $\phi$-equivariant setting, and it extends \cite[Lemma 3.1]{CS20} to the fractional setting.  For $\varphi\in C_c^\infty(\Omega)$ let
\begin{equation}\label{eq:def_phi}
\varphi_\phi(x):=\frac{1}{\mu(G)}\int_{G}\phi(g)\varphi(gx) \d{\mu},
\end{equation}
where $\mu$ is the Haar measure on G. In particular, $\varphi_\phi\in C_c^\infty(\Omega)^\phi$.

\begin{lem}\label{lem:deriv_func_equiv}Let $m\in\N_0$ and $\sigma\in[0,1]$ such that $s:=m+\sigma>0$.  If $u\in D_0^s(\Omega)^\phi$, then
\begin{equation*}
J_s^\prime(u)\varphi_\phi=J_s^\prime(u)\varphi \quad\textnormal{for every } \varphi\in C_c^\infty(\Omega).
\end{equation*}
Moreover, if $J_s^\prime(u)\vartheta=0$ for every $\vartheta\in C_c^\infty(\Omega)^\phi$, then $J_s^\prime(u)\varphi=0$ for every $\varphi\in C_c^\infty(\Omega)$.
\end{lem}
\begin{proof} We show the case $m\in\N_0$ even and $\sigma\in(0,1)$. The other cases follow analogously.  First, notice that $(-\Delta)^{\frac{m}{2}}(v\circ g)=(-\Delta)^{\frac{m}{2}}v\circ g$ for every $v\in D_0^s(\Omega)$ and $g\in G$. So, if $u$ is $\phi$-equivariant, we have that $(-\Delta)^{\frac{m}{2}}u$ is $\phi$-equivariant too. Also, 
\begin{equation}\label{eq:iden_varphi_phi}
(-\Delta)^{\frac{m}{2}}(\varphi_\phi)(x)=\frac{1}{\mu(G)}\int_{G}(-\Delta)^{\frac{m}{2}}(\phi(g)\varphi\circ g)(x)\d{\mu}=\frac{1}{\mu(G)}\int_{G}\phi(g)(-\Delta)^{\frac{m}{2}}\varphi(gx)\d{\mu}
\end{equation}
and
\begin{align}
J_s^\prime(u)\varphi_{\phi} 
&=\cE_\sigma((-\Delta)^{\frac{m}{2}}u,(-\Delta)^{\frac{m}{2}}\varphi_\phi) -\int_{\Omega}|u(x)|^{p-2}u(x)\varphi_{\phi}(x)\d{x}=: J_1+J_2 \label{eq:deriv_phi_phi}.
\end{align}
For $J_2$, we use that $u$ is $\phi$-equivariant to obtain that
\begin{align}
J_2=\frac{1}{\mu(G)}\int_{\Omega}\int_{G}|u(x)|^{p-2}u(x)\phi(g)\varphi(gx)\d{\mu}\d{x} 
= \int_{\Omega}|u(y)|^{p-2}u(y)\varphi(y)\d{y}.\label{eq:J2_term}
\end{align}
For $J_1$ we argue as follows. Since $\varphi_\phi\in C_c^\infty(\Omega)^\phi$, we use~\eqref{eq:iden_varphi_phi} to obtain that
\begin{align*}
&\int_{\R^N}\int_{\R^N}\frac{[(-\Delta)^{\frac{m}{2}}u(x)-(-\Delta)^{\frac{m}{2}}u(y)] [\int_{G}\phi(g)((-\Delta)^{\frac{m}{2}}\varphi(gx)-(-\Delta)^{\frac{m}{2}}\varphi(gy))\,d\mu]}{|x-y|^{N+2\sigma}}\d{x}\d{y} \\
&=\int_{\R^N}\int_{\R^N}\int_{G}\frac{\left[(-\Delta)^{\frac{m}{2}}u(gx)-(-\Delta)^{\frac{m}{2}}u(gy)\right]\left[(-\Delta)^{\frac{m}{2}}\varphi(gx)-(-\Delta)^{\frac{m}{2}}\varphi(gy)\right]}{|x-y|^{N+2\sigma}}\d{\mu}\d{x}\d{y}.
\end{align*}
By setting the change of variable $\bar{x}=gx$ (resp. $\bar y=g y$) and using Fubini's theorem, we have that, for every $g\in G$,
\begin{align}\notag
J_1&=\frac{c_{N,\sigma}}{\mu(G)}\int_{G}\int_{\R^N}\int_{\R^N}\frac{[(-\Delta)^{\frac{m}{2}}u(\bar{x})-(-\Delta)^{\frac{m}{2}}u(\bar{y})][(-\Delta)^{\frac{m}{2}}\varphi(\bar{x})-(-\Delta)^{\frac{m}{2}}\varphi(\bar{y})]}{|g^{-1}(\bar{x}-\bar{y})|^{N+2\sigma}}\d{ x}\d{ y} \d{\mu}\\ \label{eq:J11_term}
& =c_{N,\sigma} \int_{\R^N}\int_{\R^N}\frac{[(-\Delta)^{\frac{m}{2}}u(\bar{x})-(-\Delta)^{\frac{m}{2}}u(\bar{y})][(-\Delta)^{\frac{m}{2}}\varphi(\bar{x})-(-\Delta)^{\frac{m}{2}}\varphi(\bar{y})]}{|\bar{x}-\bar{y}|^{N+2\sigma}}\d{\bar x}\d{\bar y}.
\end{align}
To conclude, it is enough to collect identities~\eqref{eq:J2_term}--\eqref{eq:J11_term} and replace into~\eqref{eq:deriv_phi_phi} to deduce that
$J_s^\prime(u)\varphi_{\phi}=J_s^\prime(u)\varphi$.  The rest of the proof follows immediately.
\end{proof}

As a consequence of the previous results, the non-trivial $\phi$-equivariant solutions of problem~\eqref{eq:frac_domain} belong to the Nehari set
\begin{equation*}\label{N:set}
\mathcal N_s^\phi(\Omega):=\left\{u\in D_0^s(\Omega)^\phi: u\neq 0,\ {\|u\|_s^2=|u|^{\ds}_{\ds}} \right\}.
\end{equation*}

Let
\begin{equation*}
c_s^{\phi}(\Omega):=\inf_{u\in \mathcal N_s^{\phi}(\Omega)} J_s(u). 
\end{equation*}
The following result gives some properties of $\mathcal N_s^\phi(\Omega)$ and $c^\phi_s(\Omega)$. 
\begin{lem}\label{lem:min_max}Let $s>0$.
\begin{enumerate}
\item[a)] There exists $a_0>0$ such that ${\|u\|_s\geq a_0}$ for every $u\in \mathcal N_s^\phi(\Omega).$
\item[b)] $\mathcal N_s^\phi(\Omega)$ is a $C^1$-Banach sub-manifold of $D_0^s(\Omega)$ and a natural constraint for $J_s$.
\item[c)] Let $\mathcal T:=\left\{\sigma\in C^0\left([0,1];D_0^{s}(\Omega)^\phi\right): \sigma(0)=0, \sigma(1)\neq 0, J_s(\sigma(1))\leq 0\right\}$. Then 
\begin{equation*}
c_s^\phi(\Omega)=\inf_{\sigma\in\mathcal T}\max_{t\in[0,1]}J_s(\sigma(t)).
\end{equation*}
\end{enumerate}
\end{lem}
\begin{proof}
The proof follows exactly as in~\cite[Lemma 2.1]{CLR18} using \Cref{thm:sobolev}. 
\end{proof}

For a $G$-invariant domain $\Omega$, let us denote by $\Omega^G$ the set of $G$-fixed points in $\Omega$, more precisely
\begin{equation*}
\Omega^G:=\left\{x\in \Omega: Gx=\{x\} \right\}.
\end{equation*}
The next result characterizes the least-energy level on domains with $G$-fixed points.
\begin{lem}\label{lem:c_infinity}Let $s>0$. If $\Omega^G\neq \emptyset$, then $c_s^\phi(\Omega)=c_s^\phi(\R^N)$.
\end{lem}
\begin{proof}
From the inclusion $D_0^s(\Omega)\subset D^s(\R^N)$, we have that $\mathcal N_s^\phi(\Omega)\subset \mathcal N_s^\phi(\R^N)$, then
\begin{align*}
 c_s^\phi(\R^N)=\inf_{\mathcal N_s^\phi(\R^N)} J_s  \leq \inf_{\mathcal N_s^\phi(\Omega)} J_s=c_s^\phi(\Omega).
\end{align*}
For the converse, consider a sequence $(\varphi_k)_{k}$ in $\mathcal N_s^\phi(\mathbb \R^N)\cap C_c^\infty(\R^N)$ such that $J_s(\varphi_k)\to c_s^\phi(\R^N)$ and let $x_0\in \Omega^G$ and $\lambda_k>0$ such that 
$\varphi^\star_k(x):=\lambda_k^{-\frac{N}{2}+s}\varphi_k\left(\lambda_k^{-1}(x-x_0)\right)$ has support in $\Omega$. As $x_0$ is a $G$-fixed point, $\varphi_k^\star$ is $\phi$-equivariant. Thus $\varphi^\star_k\in\mathcal{N}^\phi(\Omega)$ and hence
$c_s^\phi(\Omega)\leq J_s(\varphi^\star_k)=J_s(\varphi_k)$ for all $k$. Letting $k\to+\infty$ we conclude that $c_s^\phi(\Omega)\leq c_s^\phi(\R^N)$. This ends the proof.
\end{proof} 

The following lemma can be found in~\cite[Lemma 2.4]{CLR18} or~\cite[Lemma 3.4]{CS20}.
\begin{lem}\label{lem:prop_G_seq}
If $G$ satisfies \textnormal{\textbf{($A_1$)}} then, for every pair of sequences $(\lambda_k)_{k\in\N}\subset(0,\infty)$ and $(x_k)_{k\in\N}\subset\R^N$, there exists $C_0>0$ and $(\xi_k)_{k\in\N}\subset\R^N$ such that, up to a subsequence, $\lambda_k^{-1}\textnormal{dist}(Gx_k,\xi_k)\leq C_0$ for all $k\in \N.$  Moreover, one of the following statements hold true: either $\xi_k\in \Omega^G$ or, for each $m\in \mathbb N$, there exist $g_1,\ldots,g_m\in G$ such that
$\lambda_k^{-1}|g_i\xi_k-g_j\xi_k|\to \infty$ as $k\to\infty$ if $i\neq j.$
\end{lem}

\subsection{Groups and homomorphism for sign-changing solutions}\label{subsec:Gphi}
In this section we present some symmetry groups and surjective homomorphisms that can be used to obtain different sign-changing $\phi$-equivariant solutions.  These groups and homomorphism were also used in~\cite[Lemma 3.2]{CLR18} and \cite[Lemma 4.2]{CS20}.

Let $r_\theta:\R^2\to\R^2$ be a rotation matrix (counterclockwise through an angle $\theta\in[0,\pi)$) and for $x=(x_1,x_2,x_3,x_4)^T\in \R^4$ let $R_\theta,\rho:\R^4\to \R^4$ be given by
\begin{align*}
 R_\theta x := 
 \begin{pmatrix}
r_\theta & 0 \\
0 & r_\theta
\end{pmatrix}
x\quad \text{ and }\quad 
\rho x := (x_3,x_4,x_1,x_2)^T.
\end{align*}
Let $\Upsilon$ be the subgroup generated by $\{R_\theta,\rho\::\: \theta\in [0,2\pi)\}$ and let $\phi:\Upsilon\to{\mathbb Z}_2$ be the homomorphism given by $\phi(R_\theta):=1$ for any $\theta\in [0,2\pi)$ and $\phi(\rho)=-1$. For $N\geq 4$, let $n:=\lfloor \frac{N}{4} \rfloor\geq 1$, $\Lambda_j:=O(N-4)$ if $j=1,\ldots,n-1$, and $\Lambda_n:=\{1\}$.  The $\Lambda_j$-orbit of a point $y\in \mathbb R^{N-4j}$ is an $(N-4j-1)$-dimensional sphere if $j=1,\ldots,n-1$, and it is a single point if $j=n$.  Define 
\begin{align}\label{G:def}
G_j:=(\Upsilon)^j\times \Lambda_j 
\end{align}
acting on $\R^N=\R^{4j}\times \R^{N-4j}$ by $(\gamma_1,\ldots,\gamma_j,\eta)(x_1,\ldots,x_j,y):=(\gamma_1x_1,\ldots,\gamma_jx_j,\eta y),$ where $\gamma_i\in\Upsilon$, $\eta\in \Lambda_j$, $x_i\in \R^4$, and let 
\begin{align}\label{phi:def}
\text{$\phi_j: G_j\to {\mathbb Z}_2$ be the homomorphism $\phi_j(\gamma_1,\ldots,\gamma_j,\eta):=\phi(\gamma_1)\cdots\phi(\gamma_j).$} 
\end{align}
The $G_j$-orbit of $(z_1,\ldots,z_j,y)$ is the product of orbits $G_j(z_1,\ldots,z_j,y) =  \Upsilon z_1 \times \cdots \times \Upsilon z_j \times \Lambda_j y.$ 

Note that $\phi_j$ is surjective and \textbf{($A_1$)}, \textbf{($A_2$)} are satisfied by $G_j$ and $\phi_j$ for each $j=1,\ldots,n$.  Moreover, if $u$ is $\phi_i$-equivariant, $v$ is $\phi_j$-equivariant with $i<j$, and $u(x)=v(x) \neq 0$ for some $x=(z_1,\ldots,z_j,y) \in \R^N$, then, as $u(z_1,\ldots,\varrho z_j,y) = u(z_1,\ldots,z_j,y)$  and $v(z_1,\ldots,\varrho  z_j,y) = -v(z_1,\ldots,z_j,y),$ we have that $u(z_1,\ldots,\varrho z_j,y) \neq v(z_1,\ldots,\varrho z_j,y)$. As a consequence, $u\neq v$, and Theorems \ref{main:thm:bdd} and \ref{main:thm:unbdd} yield the existence of at least $\lfloor \frac{N}{4} \rfloor$ nonradial sign-changing solutions, where $\lfloor x \rfloor$ denotes the greatest integer less than or equal to $x$.

\medskip

We refer to~\cite[Remark  4.3]{CS20} for an example of a $\phi_j$-equivariant function and for an explanation on why a similar construction is impossible for $N=1,2,3.$

\section{Uniform bounds and asymptotic estimates}\label{sec:bk}

Let $\Omega\subset\R^N$ be a smooth bounded domain.

\begin{lem} \label{A:l} Let $s>0$, $\delta\in(0,s)$, $s_k\in (s-\frac{\delta}{2},s+\frac{\delta}{2})$, and let $u_k\in D_0^{s_k}(\Omega)$ for $k\in\N$. There is $C>0$ depending only on $s,$ $\Omega$, and $\delta$ such that $\norme{u_k}_{s-\delta}\leq C \norme{u_k}_{s_k}$ for all $k\in\N.$
\end{lem}
\begin{proof}
We argue as in~\cite[Lemma 5.1]{JSW20}. By \eqref{eq:equiv_fourier}, \eqref{def:norm_s},
\begin{align*}
\|u_{k}\|^2_{s-\delta}&=\cE_{s-\delta}(u_{k},u_{k})=\int_{\R^N}|\xi|^{2(s-\delta)}|\widehat{u}_{k}|^2\d{\xi}\leq \eps^{2(s-\delta)}\norme{u_{k}}^2_{L^2(\R^N)}+\int_{|\xi|\geq \eps}|\xi|^{2(s-\delta)}|\widehat{u}_{k}|^2\d{\xi}
\end{align*}
for any $\eps\in(0,1]$, where $\widehat{u}_{k}$ is the Fourier transform of $u_{k}$. Then, using that $s_k<s+\frac{\delta}{2}$,
\begin{align}
\|u_{k}\|^2_{s-\delta}&\leq \eps^{2(s-\delta)}\norme{u_{k}}^2_{L^2(\R^N)}+\epsilon^{2(s-\delta-s_k)}\int_{\R^N}|\xi|^{2s_k}|\widehat{u}_{k}|^2\d{\xi} 
\leq \epsilon^{2(s-\delta)}\norme{u_{k}}^2_{L^2(\R^N)}+\eps^{-\delta}\norme{u_{k}}_{s_k}^2.\label{eq:est_usk:l}
\end{align}
Since $s_k>s-\frac{\delta}{2}$ we have, by Theorem~\ref{thm:rellich_type}, that $u_k\in D_0^{s-\delta}(\Omega)$ and, by the fractional Poincar\'e inequality (see e.g.~\cite[Proposition 3.3]{AJS18_loss}) and \eqref{eq:norm_equiv_h0s}, there exists $C_0>0$ only depending on $s$, $\delta$, and $\Omega$ such that $|u_{k}|^2_2\leq C_0 \|u_{k}\|^2_{s-\delta}$.  Fix $\eps=\min\{1,(\tfrac{1}{2C_0})^{\frac{1}{2(s-\delta)}}\}$, then, by \eqref{eq:est_usk:l}, $\norme{u_{k}}_{s-\delta}^2\leq C \norme{u_{k}}_{s_k}^2$, where $C=2\eps^{-\delta}$ depends only on $s,$ $\Omega$, and $\delta$.  
\end{proof}

\begin{lem}\label{aship:lem}
Let $(s_k)_{k\in\N}\subset (0,\infty)$ be such that $s_k\to s=m+\sigma$ as $k\to \infty$ with $m\in\N_0$ and $\sigma\in(0,1]$.  Let $w_k\in D^{s_k}(\R^N)$ be such that
\begin{align}\label{bd}
\|w_k\|_{s_k}<C\qquad \text{ for all $k\in\N$ and for some $C>0$},
\end{align}
then, up to a subsequence, there is $w\in D^{s}(\R^N)$ such that
 \begin{align}\label{aship}
 \eta w_k\to \eta w\qquad \text{ in $D^{s-\delta}(\R^N)$ as $k\to\infty$ for all $\eta\in C^\infty_c(\R^N)$ and all $\delta\in(0,s]$.}
 \end{align}
 In particular, for $p\in[1,2^\star_s)$,
\begin{align}\label{app}
w_k \to w\ \textnormal{in } L^p_{loc}(\R^N),\quad 
w_k \to w\ \textnormal{a.e. in } \R^N, \quad
w_k \to w\ \textnormal{in } H^m_{loc}(\R^N).
\end{align}
\end{lem}
\begin{proof}
 Let $C$, $\delta$, $w_k$ be as in the statement, $\eta\in C^\infty_c(\R^N)$, and $K:=\operatorname{supp}\eta$. 
In the following $M>0$ denotes possibly different constants depending at most on $C$, $N$, $s$, $\delta$, and $\eta$. Then, by Lemmas~\ref{A:l} and \ref{lem:prod_co}, up to a subsequence, $\|\eta w_k\|_{s-\frac{\delta}{2}} \leq M\|\eta w_k\|_{s_k}\leq M$ for all $k\in\N$.  By Theorem~\ref{thm:rellich_type}, up to a subsequence, $\eta w_k \to \eta w$  in $D^{s-\delta}(\R^N)$ as $k\to\infty$ and \eqref{aship} follows.  Moreover, by \eqref{eq:equiv_fourier} and Fatou's Lemma, 
\begin{equation}\label{fatou}
\|w\|^2_{s}=
\int_{\R^N}|\xi|^{2s}|\widehat w(\xi)|^2\d{\xi}
\leq \liminf_{k\to \infty}
\int_{\R^N}|\xi|^{2{s_k}}|\widehat w_{k}(\xi)|^2\d{\xi}
=\liminf_{k\to\infty}\|w_{k}\|_{{s_k}}^2 < C,
\end{equation}
and therefore $w\in D^s(\R^N)$. The convergence \eqref{app} follows from Theorem~\ref{thm:rellich_type}.
\end{proof}

\begin{lem}\label{lem:sigma}
For every $k\in\N$ let $\sigma_k\in(0,1)$ and $w_k\in D^{\sigma_k}(\R^N)$ be
such that $\lim_{k\to\infty}\sigma_k=:\sigma\in[0,1]$, $\|w_k\|_{\sigma_k}<C$ for all $k\in\N$ and for some $C>0$, and
\begin{align}\label{l2loc}
 w_k\to 0\qquad \text{ in $L^2_{loc}(\R^N)$ as $k\to\infty$}.
 \end{align}
Then, up to a subsequence,
 \begin{align*}
\|w_k\eta\|^2_{\sigma_k}\leq \cE_{\sigma_k}(w_k,\eta^2w_k)+o(1)\quad \text{ as }k\to\infty \text{ for all }\eta\in C^\infty_c(\R^N).
 \end{align*}
\end{lem}
\begin{proof}
Let $\eta\in C^\infty_c(\R^N)$, then
\begin{align}\label{part0}
 \|\eta w_k\|_{\sigma_k}^2-\cE_{\sigma_k}(w_k,\eta^2w_k)
 =\frac{c_{N,{\sigma_k}}}{2}\int_{\R^N}\int_{\R^N}\frac{w_k(x)w_k(y)|\eta(x)-\eta(y)|^2}{|x-y|^{N+2{\sigma_k}}}\ \d{x}\d{y}.
\end{align}
Let $K$ be the support of $\eta$ and let
$U:=\{x\in\R^N:\operatorname{dist}(x,K)\leq 1\}$. In the following $C>0$ denotes possibly different constants depending at most on $N$ and $\eta$.
By Fubini's theorem, Cauchy-Schwarz inequality, and \eqref{l2loc},
\begin{align}
&c_{N,{\sigma_k}}\left|\int_{\R^N\setminus U}\int_{\R^N}\frac{w_k(x)w_k(y)|\eta(x)-\eta(y)|^2}{|x-y|^{N+2{\sigma_k}}}\d{y}\d{x}\right|\leq c_{N,{\sigma_k}}\int_{K}|w_k(y)|\int_{\R^N\setminus U}\frac{|w_k(x)||\eta(y)|^2}{|x-y|^{N+2{\sigma_k}}}\d{x}\d{y}\notag\\
&\qquad\leq c_{N,{\sigma_k}}\left(\int_{K}|w_k|^2\right)^\frac{1}{2}
\left(\int_{K}\left(\int_{\R^N\setminus U}|w_k(x)|\frac{|\eta(y)|^2}{|x-y|^{N+2{\sigma_k}}}\d{x}\right)^2\d{y}\right)^\frac{1}{2}=o(1)\quad \text{ as }k\to\infty,\label{part1}
\end{align}
where we used that $c_{N,{\sigma_k}}$ is uniformly bounded by \eqref{cnsigma} and that, by H\"older's inequality, Theorem \ref{thm:sobolev}, and the bound $\|w_k\|_{\sigma_k}<C$,
\begin{align*}
 &\int_{K}\left(\int_{\R^N\setminus U}|w_k(x)|\frac{|\eta(y)|^2}{|x-y|^{N+2{\sigma_k}}}\d{x}\right)^2\d{y}\leq C |w_k|_{2^\star_{\sigma_k}}^2
 \int_{K}\left(\int_{\R^N\setminus U}|x-y|^{-2N}\d{x}\right)^\frac{N+2{\sigma_k}}{N}\d{y}<C.
\end{align*}
 On the other hand, by Fubini's theorem,
\begin{align}
&\int_{U}w_k(x)\int_{\R^N}w_k(y)\frac{|\eta(x)-\eta(y)|^2}{|x-y|^{N+2{\sigma_k}}}\d{y}\d{x}
=\int_{\R^N}w_k(y)\int_{U}w_k(x)\frac{|\eta(x)-\eta(y)|^2}{|x-y|^{N+2{\sigma_k}}}\d{x}\d{y}\notag\\
&=\int_{\R^N\backslash U}w_k(y)\int_{K}w_k(x)\frac{|\eta(x)|^2}{|x-y|^{N+2{\sigma_k}}}\d{x}\d{y}
+\int_{U}w_k(y)\int_{U}w_k(x)\frac{|\eta(x)-\eta(y)|^2}{|x-y|^{N+2{\sigma_k}}}\d{x}\d{y}.\label{part12}
\end{align}

The first summand is $o(1)$ as in \eqref{part1}. For the second summand, by the mean value theorem,
 \begin{align*}
I&:= c_{N,{\sigma_k}}\int_{U}\int_{U}\frac{w_k(y)w_k(x)|\eta(x)-\eta(y)|^2}{|x-y|^{N+2{\sigma_k}}}\d{x}\d{y}\leq Cc_{N,{\sigma_k}}\int_{U\times U}\frac{w_k(y)w_k(x)}{|x-y|^{N-2(1-{\sigma_k})}}\d{(x,y)}.
\end{align*}
Using the Hardy-Littlewood-Sobolev inequality (see~\cite[Theorem 4.3]{ll01} with $\eps=2(1-{\sigma_k})$, $\lambda=N-\eps$, $p=r=\frac{2}{\frac{\eps}{N}+1}<2$), \eqref{cnsigma}, and the boundedness of $U$,
\begin{align}
I&\leq Cc_{N,{\sigma_k}}\left( 
\pi^{\frac{N-2(1-{\sigma_k})}{2}}\frac{\Gamma(\frac{N}{2}-\frac{N-2(1-{\sigma_k})}{2})}{\Gamma(N-\frac{N-2(1-{\sigma_k})}{2})}
\left(\frac{\Gamma(\frac{N}{2})}{\Gamma(N)}
\right)^{\frac{N-2(1-{\sigma_k})}{N}-1}
\right) \left( \int_{U} |w_k|^p \right)^{1/p}\notag\\
&\leq C\sigma_k(1-\sigma_k)
\Gamma(1-{\sigma_k})
\left( \int_{U} |w_k|^p \right)^{1/p}\leq C\left(\int_{U}|w_k|^2\right)^{1/2}=o(1)\qquad \text{ as }k\to\infty.\label{part2}
\end{align}
The claim now follows from \eqref{part0}, \eqref{part1}, \eqref{part12}, and \eqref{part2}.
\end{proof}

To estimate all lower-order terms, we use the following.

\begin{lem}\label{lem:alpha}
 Let $(\sigma_k)_{k\in\N}\subset (0,1)$, $m\in\N$, and $\sigma\in[0,1]$ such that  $s_k:=m+\sigma_k\to s:=m+\sigma>0$ as $k\to \infty$.  For $k\in\N$, let $w_k\in D^{s_k}(\R^N)$ be such that \eqref{bd} holds and
\begin{align}\label{aship2}
 w_k\to 0\qquad \text{ pointwisely in $\R^N$ as $k\to\infty$.}
 \end{align} 
Let $\alpha\in \N^N_0$ be a multi-index such that $|\alpha|<m$, then, up to a subsequence,
 \begin{align*}
  \|\psi\partial^\alpha w_k \|_{\sigma_k} = o(1)\quad \text{ as }k\to\infty \text{ for all $\psi\in C^\infty_c(\R^N)$}.
 \end{align*}
\end{lem}
\begin{proof}
Let $\psi\in C^\infty_c(\R^N)$, $K=\operatorname{supp}(\psi)$, and let $C>0$ denote possibly different constants depending at most on $N$, $m$, $\sigma,$ and $\psi$. By \eqref{aship2} and Lemma~\ref{aship:lem},
\begin{align}\label{c}
\|w_k\|_{H^m(K)}\to 0\quad \text{ as }k\to\infty,
\end{align}
where $\|\cdot\|_{H^m(K)}$ denotes the usual norm in the Sobolev space $H^m(K)$.

 The claim now follows from the interpolation inequality (see for example~\cite[Theorem 1]{bm18} using $s_2=1,$ $s_1=0$, $p_1=p_2=p=2$, $s=\sigma_k$, $\theta=1-\sigma_k$) because
 \begin{align*}
  \|\psi\partial^\alpha w_k \|_{\sigma_k}\leq 
  |\psi\partial^\alpha w_k |^{1-\sigma_k}_{2}
  \|\psi\partial^\alpha w_k \|^{\sigma_k}_{H^1(K)}=o(1)\quad \text{as $k\to\infty$},
 \end{align*}
 since, by \eqref{c}, $|\psi\partial^\alpha w_k |^2_{2}\leq C \int_K |\partial^\alpha w_k |^2=o(1)$ as $k\to\infty$ and 
\begin{align*}
 \|\psi\partial^\alpha w_k \|^2_{H^1(K)}&= \int_K |\nabla(\partial^\alpha w_k \psi)|^2
 \leq C\int_K |\nabla\partial^\alpha w_k|^2+|\partial^\alpha w_k|^2\leq C \|w_k\|_{H^m(K)}+o(1)=o(1)\ \text{as $k\to\infty$.}\qedhere
\end{align*}
\end{proof}

\begin{lem}\label{lem:bk}
For $k\in\N$, let $(s_k)_{k\in\N}\subset(0,\infty)$ bounded and $w_k\in D^{s_k}(\R^N)$ be such that \eqref{bd} and \eqref{aship2} hold.  Then, for any $\eps\in(0,1)$, up to a subsequence,
 \begin{align*}
\|w_k\varphi\|^2_{s_k}\leq (1+\eps)\cE_{s_k}(w_k,\varphi^2w_k)+o(1)\quad \text{ as }k\to\infty\text{ for all }\varphi\in C^\infty_c(\R^N).
 \end{align*}
\end{lem}
\begin{proof} 
Since $(s_k)_{k\in\N}$ is bounded, passing to a subsequence, there is $m\in\N_0$ and $(\sigma_k)_{k\in\N}\subset [0,1]$ such that $\lim_{k\to\infty}\sigma_k=:\sigma\in[0,1]$ and $s_k=m+\sigma_k\to s=m+\sigma\geq 0$ as $k\to\infty$.

Assume first that $s_k\in(m,m+1)$ and $m$ is even. Observe that $\Delta^\frac{m}{2} (w_k\varphi) = \varphi \Delta^\frac{m}{2} w_k + R_k,$ where $R_k$ is a sum of products with derivatives of $w_k$ of order smaller than $m$. Then, by Cauchy's inequality, for $\eps>0$ arbitrarily small there is $C(\eps)>0$ such that
\begin{align*}
 |\Delta^\frac{m}{2} (w_k\varphi)(x)-\Delta^\frac{m}{2} (w_k\varphi)(y) |^2
 &=|\varphi(x) \Delta^\frac{m}{2} w_k(x) 
 -\varphi(y) \Delta^\frac{m}{2} w_k(y) + R(x) -R(y)|^2\\
 &\leq 
  (1+\eps)|\varphi(x) \Delta^\frac{m}{2}w_k(x) 
 -\varphi(y) \Delta^\frac{m}{2} w_k(y)|^{2} 
 + C(\eps)|R(x) -R(y)|^{2}
 \end{align*}
and therefore, by Lemma~\ref{lem:alpha},
\begin{align}
\|w_k\varphi\|_{s_k}^2= \|\Delta^\frac{m}{2} (w_k\varphi)\|_{\sigma_k}^2
 &\leq (1+\eps)\|\varphi\Delta^\frac{m}{2}w_k\|_{\sigma_k}^2
 +C(\eps)\|R\|_{\sigma_k}^2 = (1+\eps)\|\varphi\Delta^\frac{m}{2}w_k\|_{\sigma_k}^2
 +o(1)\label{s1}
 \end{align}
as $k\to \infty.$ Moreover, by Lemma~\ref{lem:sigma},
\begin{align}
\|\varphi\Delta^\frac{m}{2}w_k\|_{\sigma_k}\leq \cE_{\sigma_k}(\Delta^\frac{m}{2}w_k,\varphi^2\Delta^\frac{m}{2}w_k)+o(1)\quad \text{ as }k\to\infty.\label{s2}
\end{align}
Observe that $(-\Delta)^\frac{m}{2}(\varphi^2 w_k)=\varphi^2(-\Delta)^\frac{m}{2}w_k + \widetilde R,$ where $\widetilde R$ has derivatives of $w_k$ with order lower than $m$. Then
\begin{align*}
 \cE_{\sigma_k}(\Delta^\frac{m}{2}w_k,\varphi^2\Delta^\frac{m}{2}w_k)=
 \cE_{\sigma_k}(\Delta^\frac{m}{2}w_k,\Delta^\frac{m}{2}(\varphi^2w_k))
 -\cE_{\sigma_k}(\Delta^\frac{m}{2}w_k,\widetilde R).
\end{align*}
By \eqref{bd}, Cauchy-Schwarz inequality, and Lemma~\ref{lem:alpha}, we have that $|\cE_{\sigma_k}(\Delta^\frac{m}{2}w_k,\widetilde R)|\leq \|w_k\|_{s_k}\|\widetilde R\|_{\sigma_k}=o(1)$ as $k\to\infty,$ and therefore
\begin{align}\label{s3}
 \cE_{\sigma_k}(\Delta^\frac{m}{2}w_k,\varphi^2\Delta^\frac{m}{2}w_k)=
 \cE_{\sigma_k}(\Delta^\frac{m}{2}w_k,\Delta^\frac{m}{2}(\varphi^2w_k))
 +o(1)\quad \text{ as }k\to\infty.
\end{align}
But then, by \eqref{s1}, \eqref{s2}, \eqref{s3},
\begin{align*}
\|w_k\varphi\|_{s_k}^2\leq (1+\eps)\cE_{\sigma_k}(\Delta^\frac{m}{2}w_k,\Delta^\frac{m}{2}(\varphi^2w_k))
 +o(1)\quad \text{ as }k\to\infty,
\end{align*}
as claimed.  

The case $s_k\in(m,m+1)$ with $m$ odd is analogous using the corresponding norms and scalar products, see \eqref{bilin:def}.  On the other hand, if $s_k=m$ for all $k\in\N$ with $m$ even, then, by Lemma~\ref{aship:lem},
\begin{align*}
\|w_k\varphi\|^2_{s_k}
&=\int_{\R^N}|(-\Delta)^{\frac{m}{2}}(w_k\varphi)|^2
=\int_{\R^N}\varphi^2|(-\Delta)^{\frac{m}{2}}w_k|^2+o(1)\\
&=\int_{\R^N}(-\Delta)^{\frac{m}{2}}w_k (-\Delta)^{\frac{m}{2}}(\varphi^2w_k)+o(1)\quad \text{ as }k\to\infty.
 \end{align*}
 The case $s_k=m$ for all $k\in\N$ with $m$ odd is analogous.  This ends the proof.
 \end{proof}

\begin{lem}\label{arg:lem}
Let $m\in\N_0$, $(\sigma_k)_{k\in\N}\subset [0,1]$, $\lim_{k\to\infty}\sigma_k=:\sigma\in[0,1]$, $s_k:=m+\sigma_k$, $s:=m+\sigma>0$, $N>2\max\{s,s_k\}$ for all $k\in\N$. Let $\Omega\subset\R^N$ be a smooth bounded $G-$invariant domain and $u_k\in D_0^{s_k}(\Omega)^\phi$ be such that
\begin{align}
 &C^{-1}<|u_{k}|_{\dsa{s_k}}<C\qquad
 \text{ for all }k\in\N \text{ and for some }C>1,\label{lb:l}\\
 &\|J'_{s_k}(u_{k})\|_{(D^{s_k}(\R^N))'}=o(1)\quad \text{ as }k\to\infty
 \label{addhip},\\
 &u_k\to 0\qquad \text{ in $D_0^{s-\delta}(\Omega)$ as $k\to\infty$ for some $\delta\in(0,\sigma)$.}\label{eq:strong_conv_u_s_k:l}
 \end{align}
Then there are sequences $(\lambda_k)_{k\in\N}\subset(0,\infty)$, $(\xi_k)_{k\in\N}\subset(\R^N)^G$, and a constant $C_1>0$ such that $\lambda_k\to 0$, $\xi_k\to \xi\in (\R^N)^G$,
\begin{equation}\label{eq:dis_xik_Omega:2:l}
\textnormal{dist}(\xi_k,\Omega)\leq C_1\lambda_k,
\end{equation}
and the rescaling
\begin{align}\label{wkdef:l}
 w_k(y):=\lambda_k^{\frac{N}{2}-s} u_k(\lambda_k y+\xi_k),\qquad y\in \R^N,
\end{align}
satisfies that, up to a subsequence,
\begin{align*}
\eta w_k\to \eta w\qquad \text{ in $D^{s-\delta}(\R^N)$ as $k\to\infty$ for all $\eta\in C^\infty_c(\R^N)$, $\delta\in(0,s),$}
 \end{align*}
 and for some $w\in D^s(\R^N)^\phi\backslash\{0\}$.
 \end{lem}
\begin{proof}
Let $C>1$ as in \eqref{lb:l}, $\kappa_{N,s}$ as in \eqref{eq:best_constant}, and let $\tau>0$ be such that
\begin{align}\label{tau:l}
\tau<\min\{(3\kappa_{N,s_k}^2)^{-\frac{N}{2s_k}},C^{-1}\}\qquad \text{ for all }k\in\N.
\end{align}
By \eqref{lb:l}, there are $(\lambda_k)_{k\in\N}\subset(0,\infty)$ and $(x_k)_{k\in\N}\subset\R^N$ such that, passing to a subsequence,
\begin{equation}\label{eq:delta_iden_0:2:l}
\sup_{x\in \R^N}\int_{B_{\lambda_k}(x)}|u_{k}|^{\dsa{s_k}}=\int_{B_{\lambda_k}(x_k)}|u_{k}|^{\dsa{s_k}}=\tau.
\end{equation}

For the chosen sequences $(\lambda_k)_{k\in\N}$ and $(x_k)_{k\in\N}$, let $C_0>0$ and $(\xi_k)_{k\in\N}$ be given by Lemma~\ref{lem:prop_G_seq}. Then, $|g_kx_k-\xi_k|\leq C_0\lambda_k$ for some $g_k\in G$ and, since $|u_{k}|$ is $G$-invariant, we have  that
\begin{equation}\label{eq:delta_iden:2:l}
\tau=\int_{B_{\lambda_k}(g_kx_k)}|u_{k}|^{\dsa{s_k}}\leq \int_{B_{C_1\lambda_k}(\xi_k)}|u_{k}|^{\dsa{s_k}},
\end{equation}
where $C_1:=C_0+1$ and, in particular, \eqref{eq:dis_xik_Omega:2:l} holds. 

We claim that $(\xi_k)_{k\in\N}\subset(\mathbb R^N)^G$. Otherwise, by \Cref{lem:prop_G_seq}, we have, for each $m\in\mathbb{N}$, $m$ different elements $g_1,\ldots,g_m\in G$ such that $B_{C_1\lambda_k}(g_i\xi_k)\cap B_{C_1\lambda_k}(g_j\xi_k)=\emptyset$ for $k$ large enough. Therefore, from~\eqref{eq:delta_iden:2:l},
\begin{equation*}
m\tau \leq \sum_{i=1}^{m}\int_{{B_{C_1\lambda_k}(g_i\xi_k)}}|u_{k}|^{\dsa{s_k}}\leq \int_{\Omega}|u_k|^{\dsa{s_k}}<C\quad \text{for every $m\in\mathbb{N}$, }
\end{equation*}
which yields a contradiction to \eqref{lb:l}. Thus, $(\xi_k)_{k\in\N}\subset(\mathbb R^N)^G$. Let $w_k$ be given by \eqref{wkdef:l}.  Since $u_{k}$ is $\phi$-equivariant and $\xi_k$ is a $G$-fixed point, we have that $w_k$ is $\phi$-equivariant. Observe that, by \eqref{eq:delta_iden_0:2:l}, \eqref{eq:delta_iden:2:l}, and a change of variables,
\begin{equation}\label{eq:delta_wk:2:l}
\tau=\sup_{z\in\R^N}\int_{B_1(z)}|w_k|^{\dsa{s_k}}\leq \int_{B_{C_1}(0)}|w_k|^{\dsa{s_k}}.
\end{equation}
Similarly, by \eqref{lb:l} and \eqref{addhip}, $(w_k)_{k\in\N}$ is uniformly bounded in $D^{s_k}(\R^N)$. By Lemma~\ref{aship:lem}, there is $w\in D^{s}(\R^N)^\phi$ such that, passing to a subsequence,  
\begin{align*}
\eta w_k \to \eta w & \textnormal{ in }D^{s-\eps}(\R^N)^\phi\text{ as }k\to\infty \text{ for all }\eta\in C^\infty_c(\R^N)\text{ and }\eps\in(0,s).
\end{align*}

Now, we prove by contradiction that $w\neq 0$. Assume that $w=0$. Given $\varphi\in C_c^\infty(\R^N)$, we set
\begin{equation*}
\vartheta(x):=\frac{1}{\mu(G)}\int_{G}\varphi^2(gx)\d{\mu} \quad\text{and}\quad \vartheta_k(x)=\vartheta\left(\frac{x-\xi_k}{\lambda_k}\right).
\end{equation*}
Using that $w_k$ is $\phi$-equivariant and according to definition~\eqref{eq:def_phi}, a direct computation yields $(\varphi^2 w_k)_\phi=\vartheta w_k$, whence $\vartheta_ku_k$ is $\phi$-equivariant and $\|\vartheta_k u_k\|_{s_k}$ is uniformly bounded. From here, using \Cref{lem:deriv_func_equiv} and \eqref{addhip},
\begin{equation}\label{eq:asymp_func_claim:l}
J_{s_k}^\prime(w_k)(\varphi^2 w_k)=J_{s_k}^\prime(w_k)\vartheta w_k=J_{s_k}^\prime(u_k)(\vartheta_k u_k)=o(1)\quad \text{as $k\to\infty$.}
\end{equation}
By Lemma~\ref{lem:bk}, $\|w_k\varphi\|^2_{s_k}\leq \frac{3}{2}\cE_{s_k}(w_k,\varphi^2w_k)+o(1)$ as $k\to\infty.$  
Let $\varphi\in C^\infty_c(B_1(z))$ with $z\in\R^N$. Then, by H\"older's inequality and \eqref{eq:asymp_func_claim:l},
\begin{align*}
\norme{w_k\varphi}_{s_k}^2 \leq \frac{3}{2}\int_{B_1(z)}|w_k|^{\dsa{s_k}-2}|w_k\varphi|^2+o(1)
&\leq \frac{3}{2}\left(\int_{B_1(z)}|w_k|^{\dsa{s_k}}\right)^{\frac{\dsa{s_k}-2}{\dsa{s_k}}}\left(\int_{\R^N}|\varphi w_k|^{\dsa{s_k}}\right)^{\frac{2}{\dsa{s_k}}}+o(1)\\
&\leq \frac{3}{2}\tau^{\frac{2s_k}{N}} |\varphi w_k|^2_{\dsa{s_k}}+o(1).
\end{align*}
By \Cref{thm:sobolev} and \eqref{tau:l},
\begin{align} \label{eq:bound_norm:l}
\norme{\varphi w_k}_{s_k}^2 \leq \frac{3}{2}\tau^{\frac{2s_k}{N}}\kappa_{N,s}^2 \norme{\varphi w_k}^2_{s_k}+o(1)
\leq \frac{1}{2}\norme{\varphi w_k}^2_{s_k}+o(1)\quad \text{as $k\to\infty$. }
\end{align}
By \eqref{eq:bound_norm:l}, we have that $\|\varphi w_k\|_{s_k}=o(1)$ and therefore (by Theorem~\ref{thm:sobolev}) $|\varphi w_k|_{\dsa{s_k}}=o(1)$ as $k\to\infty$ for any $\varphi\in C_c^\infty(B_1(z))$, which contradicts~\eqref{eq:delta_wk:2:l}. Therefore,
\begin{align}\label{w_claim}
 w\neq 0\quad \text{ in }\R^N.
\end{align}
Then, passing to a subsequence, $\xi_k\to \xi\in (\R^N)^G$ as $k\to\infty$ and, by \eqref{eq:strong_conv_u_s_k:l}, \eqref{wkdef:l}, and \eqref{w_claim}, we conclude that $\lambda_k\to 0$ as $k\to\infty$.
\end{proof}

\section{A concentration result}\label{sec:c}

In this section, we show a concentration result following the strategy from \cite[Theorem 2.5]{CLR18} and \cite[Theorem~3.5]{CS20} (see also~\cite[Theorem 8.13]{Will96}). Recall that $A^G$ denotes the set of $G$-fixed points of $A\subset\R^N$.

\begin{theo}\label{thm:concentration}
Assume that $G$ and $\phi$ satisfy $(A_1)$ and $(A_2)$. Let $s>0$, $N\geq 1$, $N>2s$, $\Omega$ be a $G$-invariant bounded smooth domain in $\R^N$, and let $u_k\in D_0^s(\Omega)^\phi$ be such that
\begin{equation}\label{eq:hyp_main_teo}
J_s(u_k)\to c_s^\phi(\Omega)\quad\text{ and }\quad  J_s^\prime(u_k)\to 0 \quad{\textnormal{in }} (D_0^s(\Omega)^\phi)^\prime\quad \text{ as }k\to\infty.
\end{equation}
Then, up to a subsequence, one of the following two possibilities occurs
\begin{enumerate}
\item[\textnormal{(I)}] $(u_k)_{k\in\N}$ converges strongly in $D_0^s(\Omega)$ to a minimizer of $J_s$ on $\mathcal N_s^\phi(\Omega)$, or
\item[\textnormal{(II)}] there exist a sequences $(\xi_k)_{k\in\N}\subset(\R^N)^G$, $(\lambda_k)_{k\in\N}\subset(0,\infty)$, and a nontrivial solution $w$ to
\begin{equation}\label{eq:main_teo}
(-\Delta)^s w=|w|^{\ds-2}w, \quad w\in D_0^s(\esp)^\phi
\end{equation}
with the following properties:
\begin{enumerate}
\item[\textnormal{(i)}] $\lambda_k\to 0$, $\xi_k\to \xi$, $\xi\in (\ov{\Omega})^G$, and $\lambda_k^{-1}\operatorname{dist}(\xi_k,\Omega)\to d\in[0,\infty]$.
\item[\textnormal{(ii)}] If $d=\infty$, then $\esp=\R^N$ and $\xi_k\in\Omega$. 
\item[\textnormal{(iii)}] If $d\in[0,\infty)$, then $\xi\in \partial\Omega$ and $\esp=\{x\in\R^N:x\cdot \nu>\ov{d}\}$, where $\nu$ is the inward-pointing  unit normal to $\partial \Omega$ at $\xi$ and $\ov{d}\in\{d,-d\}$. Moreover, $\esp$ is $G$-invariant, $\esp^G\neq 0$, and $\Omega^G\neq 0$. 
\item[\textnormal{(iv)}] $w\in \cN_s^\phi(\esp)$ and $J_s(w)=c_s^\phi(\R^N)$.
\item[\textnormal{(v)}] $\lim\limits_{k\to \infty}\norme{u_k-\lambda_k^{-\frac{N}{2}+s}w\left(\frac{\cdot-\xi_k}{\lambda_k}\right)}_s=0$.
\end{enumerate}
\end{enumerate}
\end{theo}

\begin{proof}
Since $J_s^\prime(u_k)u_k=\|u_k\|_s^2-|u_k|^{\ds}_{\ds}$, we have, by \eqref{eq:hyp_main_teo}, that
\begin{align}
\frac{s}{N}\|u_k\|_s^2=J_s(u_k)-\frac{1}{\ds}J_s^\prime(u_k)u_k
\leq C+o(1)\|u_k\|_s. \label{eq:equiv_Juk}
\end{align}
Therefore, $(u_k)_{k\in\N}$ is bounded in $D_0^s(\Omega)^\phi$ and, up to a subsequence, there exists $u\in D_0^s(\Omega)^\phi$ such that
\begin{equation}\label{weak}
u_k \weakly u \quad\text{weakly in } D_0^s(\Omega)^\phi \text{ as }k\to\infty. 
\end{equation}
By \Cref{thm:rellich_type}, up to a subsequence, $u_k\to u$ strongly in $L^\nu(\Omega)$ for any $\nu\in[2,\ds)$. Then,
\begin{equation*}
\lim_{k\to\infty}\int_{\Omega}{|u_k|}^{\ds-2}u_k\varphi=\int_{\Omega}|u|^{\ds-2}u\varphi \quad \text{ for all } \varphi\in C_c^\infty{(\Omega)}^\phi. 
\end{equation*}
Hence, for any $\varphi\in C_c^\infty(\Omega)^\phi$,
\begin{align}
J_s^\prime(u)\varphi&=\cE_s(u,\varphi)-\int_{\Omega}|u(x)|^{\ds-2}u(x)\varphi(x)\d{x} \notag\\ &= \lim_{k\to\infty}\left[\cE_s(u_k,\varphi) - \int_{\Omega}|u_k(x)|^{\ds-2}u_k(x)\varphi(x)\d{x} \right]=\lim_{k\to\infty} J_s^\prime(u_k)\varphi=0.\label{eq:lim_func}
\end{align}

We now consider two cases:

\medskip

\noindent(I) If $u\neq 0$ then $u\in \mathcal N_s^\phi(\Omega)$ and, by \eqref{eq:hyp_main_teo} and~\eqref{eq:equiv_Juk},
\begin{align*}
c_s^\phi(\Omega)&\leq J_s(u)=\frac{1}{2}\|u\|_s^2-\frac{1}{\ds}|u|^{\ds}_{\ds}=\frac{s}{N}\norme{u}_s^2\leq \liminf_{k\to \infty} \frac{s}{N} \norme{u_k}_s^2=c_s^\phi(\Omega)+o(1),
\end{align*}
This with \eqref{weak} implies that $u_k\to u$ strongly in $D_0^s(\Omega)^\phi$ and then $J_s(u)=c_s^\phi(\Omega)$.

\medskip

\noindent (II) If $u=0$, then \eqref{eq:hyp_main_teo} and \eqref{eq:equiv_Juk} imply that
\begin{equation*}
\int_{\Omega}|u_k(x)|^{\ds}\dx=\frac{N}{s}\left(J_s(u_k)-\frac{1}{2}J_s^\prime(u_k)u_k\right)\to \frac{N}{s}c_s^\phi(\Omega)\quad \text{ as }k\to\infty.
\end{equation*}
Note that $u_k$ satisfies the assumptions of Lemma~\ref{arg:lem} (with $s_k=s$ for all $k\in\N$). Let $\lambda_k\to 0$, $\xi_k\to \xi$, $w_k$, and $w$ as given by Lemma~\ref{arg:lem} and define $d:=\lim_{k\to\infty} \lambda_k^{-1}\operatorname{dist}(\xi_k, \partial \Omega)\in [0,\infty],$ $\Omega_k:=\{y\in\R^N\::\: \lambda_k y + \xi_k\in\Omega\}$.

If $d=\infty$, then, by~\eqref{eq:dis_xik_Omega:2:l}, we have that $\xi_k\in \Omega$. Hence, for every $X\subset\subset \R^N$, there exists $k_0$ such that $X\subset \Omega_k$ for all $k\geq k_0$. Thus, for $d=\infty$ we set $\esp:=\R^N$.  Otherwise, if $d\in[0,\infty)$, then, as $\lambda_k\to 0$, we have that $\xi\in\partial \Omega$. If a subsequence of $(\xi_k)$ is contained in $\ov{\Omega}$, we set $\ov{d}:=-d$, otherwise we take $\ov{d}:=d$. We define $\mathbb H:=\left\{y\in \R^N:y\cdot\nu>\ov{d}\right\},$ 
where $\nu$ is the inward-pointing unit normal to $\partial \Omega$ at $\xi$. Since $\xi$ is a $G$-fixed point so is $\nu$. Thus $\Omega^G\neq \emptyset$, $\mathbb H$ is $G$-invariant and $\mathbb{H}^G\neq \emptyset$. If $X$ is compact and $X\subset \mathbb H$, there exists $k_0$ such that $X\subset \Omega_k$ for all $k\geq k_0$. Moreover, if $X$ is compact and $X\subset \R^N\setminus \ov{\mathbb H}$, then $X\subset\R^N\setminus \Omega_k$ for $k$ large enough. As $w_k\to w$ a.e. in $\R^N$, this implies that $w=0$ a.e. in $\R^N\setminus\mathbb{H}$. So $w\in D_0^s(\mathbb H)^\phi$. Then, for $d<\infty$, we set $\esp:= \mathbb H$.

For $\varphi\in C_c^\infty(\esp)^\phi$, we define $\varphi_k(x):=\lambda_k^{-\frac{N}{2}+s}\varphi\left(\frac{x-\xi_k}{\lambda_k}\right)$.  Since $\xi_k$ is a $G$-fixed point, then $\varphi_k$ is $\phi$-equivariant and there is $k_0$ large enough such that $\operatorname{supp }\varphi_k\subset\Omega$. By \eqref{issues}, $\varphi_k$ is uniformly bounded in $D_0^s(\Omega)$; hence,~\eqref{eq:lim_func}, Lemmas~\ref{lem:deriv_func_equiv} and~\ref{arg:lem}, and a direct computation yield $J_s^\prime(w_k)\varphi=J_s^\prime(u_k)\varphi_k=o(1)$ as $k\to\infty.$ Therefore $w$ is a nontrivial weak solution of~\eqref{eq:main_teo}. From Lemma~\ref{lem:c_infinity}, we conclude that $c_s^\phi(\Omega)=c_s^\phi(\esp)=c^\phi_s(\R^N)$. Hence, 
\begin{equation*}
c_s^\phi(\R^N)\leq J_s(w)=\frac{s}{N}\|w\|_s^2 \leq \liminf_{k\to \infty} \frac{s}{N} \|w_k\|_s^2=\frac{s}{N}\liminf_{k\to\infty}\|u_k\|_s^2=c_s^\phi(\R^N).
\end{equation*}

Thus, $J_s(w)=c_{\infty}^\phi$ and $w_k\to w$ strongly in $D^s(\R^N)$ as $k\to\infty$. By a change of variable, this implies that
\begin{equation*}
o(1)=\norme{w_k-w}_s=\norme{u_k-\lambda_k^{-\frac{N}{2}+s}w\left(\frac{\cdot-\xi_k}{\lambda_k}\right)}_s\quad \text{as $k\to\infty$}.
\end{equation*}
This ends the proof. 
\end{proof}

\section{Existence, nonexistence, and convergence of solutions in symmetric bounded domains}\label{sec:bdd}

We begin this section with the proof of the nonexistence result stated in the introduction.

\begin{proof}[Proof of Proposition \ref{prop:nonex}]
 By contradiction, let $u$ be a nontrivial nonnegative solution of \eqref{nonex:eq}. Let 
 \begin{align*}
{\mathscr{D}}^s u(z):=\lim_
{\substack{x\to z \\x\in B, \frac{x}{|x|}=z}}
\frac{u(x)}{(1-|x|^2)^s}
=\frac{1}{2^s}\lim_
{\substack{x\to z \\x\in B, \frac{x}{|x|}=z}}
\frac{u(x)}{(1-|x|)^s}
=\frac{1}{2^s}\lim_
{\substack{x\to z \\x\in B, \frac{x}{|x|}=z}}
\frac{u(x)}{\operatorname{dist}(x,\R^N\backslash B)^s}.
\end{align*}
By \cite[Corollary 1.9 and Lemma 2.1]{AJS18_poisson}, we have that
\begin{align*}
{\mathscr{D}}^s u(z)=\frac{\Gamma(\frac{N}{2})}{4^{s}\pi^{\frac{N}{2}}\Gamma(s)^2s}\int_B \frac{(1-|y|^2)^s}{|y-z|^N}u(y)^{\ds-1}\d{y}>0\quad \text{ for }z\in \partial B.
\end{align*}
However, by the Pohozaev identity \cite[Corollary 1.7]{rs15}, $\int_{\partial B}|{\mathscr{D}}^s u(z)|^2\ d\sigma=0.$  This yields a contradiction, and therefore \eqref{nonex:eq} has no nontrivial nonnegative solutions.
\end{proof}

For bounded domains without fixed points, we have the following existence result.
\begin{prop}\label{cor:nonfixed}
Assume that $G$ and $\phi$ verify assumptions \textnormal{\textbf{($A_1$)}} and \textnormal{\textbf{($A_2$)}}. Let $\Omega$ be a $G$-invariant bounded smooth domain in $\R^N$ such that $\overline{\Omega}^G=\emptyset$ and let $s>0$. Then, the problem
\begin{equation}\label{eq:cor-crit}
\begin{cases}
(-\Delta)^su=|u|^{\ds-2}u, \\
u\in D_0^s(\Omega)^\phi,
\end{cases}
\end{equation}
has a least-energy solution. The solution is sign-changing if $\phi:G\to \mathbb Z_2$ is surjective.
\end{prop}
\begin{proof}
By a), c) of Lemma~\ref{lem:min_max} and \cite[Thm. 2.9]{Will96} there is a sequence $(u_k)$ such that~\eqref{eq:hyp_main_teo} holds. Then, by \Cref{thm:concentration}, since alternative (II) cannot hold due to the lack of fixed points in $\overline{\Omega}$, we conclude that $J_s$ attains a minimum $u\in\mathcal N_s^\phi(\Omega)$. Then there is a Lagrange multiplier $\lambda\in \R$ such that 
\begin{align}\label{lm}
 J_s'(u)\varphi = \lambda\left(2\cE_s(u,\varphi)-2^\star_s\int_\Omega |u|^{2^\star_s-2}u\varphi \right) \qquad \text{ for all }\varphi\in D^s_0(\Omega)^\phi.
\end{align}
Testing with $\varphi=u$, we obtain that 
$0=(1-2\lambda)\|u\|_s^2+(2^\star_s\lambda-1)|u|^{2^\star_s}_{2^\star_s}
 =(1-2\lambda+2^\star_s\lambda-1)|u|^{2^\star_s}_{2^\star_s}
 =(2^\star_s-2)\lambda|u|^{2^\star_s}_{2^\star_s}.$
Since $u\neq 0$, this implies that $\lambda=0$. Then \eqref{lm} and Lemma \ref{lem:deriv_func_equiv} imply that $u$ is a weak solution of \eqref{eq:cor-crit}.
\end{proof}

Next, we show some convergence properties of the solutions to~\eqref{eq:cor-crit}.  We begin with an auxiliary lemma.

\begin{lem}\label{lem:bds} Let $G$ and $\phi$ verify \textnormal{\textbf{($A_1$)}}, \textnormal{\textbf{($A_2$)}} and let $\Omega$ be a $G$-invariant bounded smooth domain in $\R^N$ such that $\overline{\Omega}^G=\emptyset$.  Let $s>0$, $N>2s,$ and $0<\delta< \min\{s,\frac{N}{2}-s\}$. For each $t\in(s-\delta,s+\delta)$, let $u_t$ be a least-energy solution to~\eqref{eq:cor-crit} given by Proposition~\ref{cor:nonfixed}. Then there is a constant $C>1$ depending only on $\delta,$ $\Omega$, and $s$ such that
\begin{align}\label{ut}
C^{-1}<\|u_{t}\|_{t}<C\qquad \text{ for all }t\in(s-\delta,s+\delta).
\end{align}
\end{lem}
\begin{proof}
 Let $N$, $s$, $\delta,$ $t$, and $u_t$ be as in the statement and let $\varphi\in C^\infty_c(\Omega)\backslash\{0\}$. Then
 \begin{align*}
k_t  \varphi\in\cN_t,\qquad k_t:=\left(\frac{\|\varphi\|_t^2}{|\varphi|_{2^\star_t}^{2^\star_t}}\right)^\frac{1}{2^\star_t-2}.
 \end{align*}
 Then, since $\dsa{t}-2>0$ for all $t\in[s-\delta,s+\delta]$,
 \begin{align*}
  c^\phi_t(\Omega)\leq J_t(k_t\varphi)\leq \sup_{t\in(s-\delta,s+\delta)}J_t(k_t\varphi)=:C_1,
 \end{align*}
 where $C_1>0$ depends only on $\varphi$, $s$, $\delta$, and $\Omega$. Moreover, since $u_t\in\cN_t^\phi$ is a least-energy solution,
 \begin{align*}
  \|u_{t}\|^2_{t}
  =\frac{N}{t}c^\phi_t(\Omega)
  \leq \frac{N}{s-\delta}C_1=:C_2.
\end{align*}
This establishes the upper bound in \eqref{ut}. To obtain the lower bound, let
\begin{equation*}
F_t(u):=\norme{u}^2_t-|u|_{\dsa{t}}^{\dsa{t}}\geq \norme{u}_t^2-\kappa^{\dsa{t}}_{N,t}\norme{u}_t^{\dsa{t}}\quad \text{ for }u\in D_0^t(\Omega),
\end{equation*}
where $\kappa_{N,t}$ is explicitly given by \eqref{eq:best_constant}. In particular, by the definition of $\delta$, $N>2t$ and therefore
\begin{align*}
 \sup_{t\in(s-\delta,s+\delta)}\kappa_{N,t}
 =\sup_{t\in(s-\delta,s+\delta)}2^{-2t}\pi^{-t}\frac{\Gamma(\frac{N-2t}{2})}{\Gamma(\frac{N+2t}{2})} \left(\frac{\Gamma(N)}{\Gamma(N/2)}\right)^{2t/N}=:K,
\end{align*}
where $K$ depends only on $N,s,$ and $\delta$. Then, for $\|u\|_t<1$,
\begin{equation*}
F_t(u)\geq  \norme{u}_t^2\left(1-K^{\dsa{t}}\norme{u}_t^{\dsa{t}-2}\right)\geq\norme{u}_t^2\left(1-(K+1)^{\dsa{s+\delta}} \norme{u}_t^{\frac{4(s-\delta)}{N-2(s-\delta)}}\right).
\end{equation*}
In particular, $F_t(u)>0$ for all $t\in(s-\delta,s+\delta)$ if
\begin{align*}
 0<\|u\|_t < (K+1)^{-\frac{N (2 \delta +N-2 s)}{2 (s-\delta ) (N-2 (\delta +s))}}=:a.
\end{align*}
Since $F_t(u_t)=0$ because $u_t\in \cN_t$, necessarily $\|u_t\|_t>a$ for $t\in (s-\delta,s+\delta).$ This yields the lower bound in \eqref{ut}.
\end{proof}

Our main convergence result is the following.

\begin{theo}\label{thm:conv:bdd}
Let $G$ and $\phi$ verify \textnormal{\textbf{($A_1$)}}, \textnormal{\textbf{($A_2$)}}, $\Omega$ be a $G$-invariant bounded smooth domain in $\R^N$ such that $\overline{\Omega}^G=\emptyset$, $m\in\mathbb N_0$, $(\sigma_k)_{k\in\N}\subset[0,1]$, $\lim_{k\to\infty}\sigma_k=:\sigma\in[0,1]$, $s_k:=m+\sigma_k>0$, $s:=m+\sigma>0$, and $N>2\max\{s,s_k\}$ for all $k\in\N$.  Let $u_{s_k}$ be a least-energy solution of
\begin{equation*}
(-\Delta)^{s_k} u_{s_k}=|u_{s_k}|^{2^\star_{s_k}-2}u_{s_k}, \qquad u_{s_k}\in D_0^{s_k}(\Omega)^\phi.
\end{equation*}
Then, up to a subsequence,
\begin{equation*}
u_{s_k}\to u \quad \textnormal{strongly in $D_0^{s-\delta}(\Omega)$ as $k\to\infty$ for all $\delta\in(0,s)$,}  
\end{equation*}
where $u$ is a least-energy solution of
\begin{equation}\label{eq:crit_limit}
(-\Delta)^{s} u=|u|^{2^\star_{s}-2}u, \qquad u\in D_0^s(\Omega)^\phi.
\end{equation}
\end{theo}
\begin{proof}
Let $s$, $s_k,$ $u_{s_k}$ be as in the assumptions and let $0<\delta<\min\{s,\frac{N}{2}-s,\frac{(N-2s)^2}{2(N+2s)}\}$.  In the following, $C>1$ denotes possibly different constants depending at most on $s$, $\delta$, $N$, and $\Omega$. Passing to a subsequence, we may assume that $s_k\in(s-\delta,s+\delta)$ for all $k\in\N$.  By Lemma~\ref{lem:bds},
\begin{align*}
 C^{-1}<\|u_{s_k}\|^2_{s_k}=|u_{s_k}|^{\dsa{s_k}}_{\dsa{s_k}}<C\qquad
 \text{ for all }k\in\N.
\end{align*}
By Lemma~\ref{A:l}, $\norme{u_{s_k}}_{s-\frac{\delta}{2}}^2\leq C \norme{u_{s_k}}_{s_k}^2<C$.  Then, by Theorem~\ref{thm:rellich_type}, there is {$u\in D_0^{s-\delta}(\Omega)$} such that
\begin{align}
u_{s_k}\to u \ \text{in $D^{s-\delta}_0(\Omega)$,}\qquad 
u_{s_k}\to u \ \text{in $L^p(\Omega)$ for $p\in[2,2^\star_{s-\delta})$}\qquad \text{as $k\to\infty$.}\label{eq:strong_conv_u_s_k}
\end{align}
Note that, since $\delta<\frac{(N-2s)^2}{2(N+2s)}$, then 
\begin{align}\label{twos}
\dsa{s}-1<\dsa{s-\delta}.
\end{align}

Moreover, using Fatou's Lemma as in \eqref{fatou}, we have that $\|u\|^2_{s}\leq \liminf_{k\to\infty}\|u_{s_k}\|_{{s_k}}^2 < C,$ and therefore $u\in D_0^s(\Omega)$.  Since $u_{s_k}$ is a least-energy solution, we have from integration by parts (see e.g.~\cite[Lemma 1.5]{AJS18}) that
\begin{align*}
0=J_{s_k}^\prime(u_{s_k})\varphi&=\cE_{s_k}(u_{s_k},\varphi)-\int_{\Omega}|{s_k}|^{2^\star_{s_k}-2}u_{s_k}\varphi=\int_{\Omega}u_{s_k}(-\Delta)^{s_k}\varphi-\int_{\Omega}|u_{s_k}|^{2^\star_{s_k}-2}u_{s_k}\varphi.
\end{align*}

Note that, by \eqref{twos}, $2^\star_{s_k}-1=2^\star_s-1+o(1)<2^\star_{s-\delta}-\eps+o(1)<2^\star_{s-\delta}$ as $k\to\infty$ with $\eps=\frac{1}{2}(2^\star_{s-\delta}-2^\star_s+1)>0$; then, by \eqref{eq:strong_conv_u_s_k}, and Lemma \ref{phi:con},
\begin{align*}
0 = \int_{\Omega} u(-\Delta)^{s}\varphi-\int_{\Omega}|u|^{2^\star_s-2}u\varphi
\qquad \text{ for all }\varphi\in C^\infty_c(\Omega),
\end{align*}
that is, $u$ is a weak solution of the limit problem \eqref{eq:crit_limit}.  Note that
\begin{align}\label{claimu}
u\neq 0 \quad \text{ in }\Omega.
\end{align}
Indeed, assume by contradiction that $u=0$. Then $u_{s_k}$ satisfies the assumptions of Lemma~\ref{arg:lem}.  Let $\lambda_k\to 0$ and $\xi_k\to\xi$ be given by Lemma~\ref{arg:lem} and define $d:=\lim_{k\to\infty} \lambda_k^{-1}\operatorname{dist}(\xi_k, \partial \Omega)\in [0,\infty].$  If $d=\infty$, then, by~\eqref{eq:dis_xik_Omega:2:l},  $\xi_k\in \Omega$. But this cannot happen since $\overline{\Omega}^G=\emptyset$ and $\xi_k\in (\R^N)^G$. On the other hand, if $d\in[0,\infty)$, then, as $\lambda_k\to 0$, we have that $\xi\in\partial \Omega$, which also cannot happen, because $\overline{\Omega}^G=\emptyset$. We have reached a contradiction and \eqref{claimu} follows.

Next, we show that $u$ is a least-energy solution, namely, that $J_{s}(u)=c_s^{\phi}(\Omega).$  By Lemma~\ref{lem:bds}, there is $C>0$ such that $C^{-1}<c_{s_k}^\phi<C$ for all $k\in\N.$ In particular, passing to a subsequence, there is $c_*$ such that $c_{s_k}^\phi\to c_*$ as $k\to\infty$.  Then, using Fatou's Lemma as in \eqref{fatou},
\begin{align}
 c_s^\phi 
 &\leq J_s(u) 
 =\left(\frac{1}{2}-\frac{1}{2_s^\star}\right)\|u\|_s^2
 \leq \liminf_{k\to\infty}\left(\frac{1}{2}-\frac{1}{2_{s_k}^\star}\right)\|u_{s_k}\|_{s_k}^2\notag\\
 &=\liminf_{k\to\infty}J_{s_k}(u_{s_k})
 =\liminf_{k\to\infty}c_{s_k}^{\phi}
 =c_*.\label{3}
\end{align}
On the other hand, by Proposition \ref{cor:nonfixed}, there is $u_s\in \cN_s$ such that $J_s(u_s)=c_s^\phi$. Then
\begin{equation}\label{1}
t_k:=\left(\frac{\|u_s\|_{s_k}^2}{|u_s|_{2^\star_{s_k}}^{2^\star_{s_k}}}\right)^{\frac{1}{2^\star_{s_k}-2}}= 1 + o(1)\quad \text{ as }k\to\infty
\end{equation}
and
\begin{equation}\label{2}
\|u_s\|_{s_k}=\|u_s\|_s+o(1), \quad |u_s|_{\dsa{s_k}}=|u_s|_{\dsa{s}}+o(1)\quad \text{ as }k\to\infty.
\end{equation}
But then, using the minimality of $u_{s_k}$, \eqref{1}, and \eqref{2},
\begin{align*}
c_*^\phi+o(1)&=c_{s_k}^\phi=J_{s_k}(u_{s_k})\leq J_{s_k}(t_{k}u_s)=J_s(u_s)+o(1)=c_s^{\phi}+o(1),\qquad \text{as $k\to\infty$.}
\end{align*}
Therefore $c_*^\phi\leq c_s^{\phi}$ and, with \eqref{3}, we conclude that $c_*^\phi=c_s^{\phi}$ and that $J_s(u)=c_s^\phi$.
\end{proof}

\begin{proof}[Proof of Theorem~\ref{main:thm:bdd}]
 The first part (existence) follows from Proposition~\ref{cor:nonfixed} and the second part (convergence) follows from Theorem~\ref{thm:conv:bdd}.
\end{proof}

 \section{Existence and convergence of entire solutions}\label{sec:ubd}
 
 Recall the definition of ${\cal N}_s^\phi$ given in \eqref{N:set} and of $J_s$ given in \eqref{eq:func_RN}.  We begin with an existence theorem.
 
 \begin{theo}\label{thm:existence_teo}
 Let $N\in\N$, $s>0,$ $N>2s$, $G$ be a closed subgroup of $O(N)$, and $\phi:G\to \mathbb{Z}_2$ be a continuous homomorphism satisfying {\bf ($A_1$)} and {\bf ($A_2$)}. Then, $J_s$ attains its minimum on $\mathcal N_s^\phi(\R^N)$. Consequently, the problem 
 \begin{equation}\label{eq6}
(-\Delta)^su=|u|^{\ds-2}u, \qquad 
u\in D^s(\R^N)^\phi,
\end{equation}
has a nontrivial $\phi$-equivariant solution.  The solution is sign-changing if $\phi$ is surjective.
 \end{theo}
 \begin{proof}
 The unitary ball $B=\{x\in\R^N:|x|<1\}$ is $G$-invariant for every subgroup $G$ of $O(N)$. Since $0\in B^{G}$ then $c_s^\phi(B)=c_s^\phi(\R^N)$, by \Cref{lem:c_infinity}. By  a) and c) of \Cref{lem:min_max} and~\cite[Thm. 2.9]{Will96}, we obtain the existence of a sequence $(u_k)_{k\in\N}\subset D_0^s(B)^\phi$  such that $J_s(u_k)\to c_s^\phi(\Omega)$ and $J_s^\prime(u_k)\to 0$ in $(D_0^s(B)^\phi)^\prime$ as $k\to\infty.$  Then, by \Cref{thm:concentration}, there exists $u\in \mathcal N_s^\phi(\mathbb E)\subset \mathcal N_s^\phi(\R^N)$ with $J_s(u)=c_s^\phi(\R^N)$, and therefore $J_s$ attains its minimum on $\mathcal N_s^\phi(\R)$.  Arguing as in Proposition~\ref{cor:nonfixed}, we conclude that $u$ is a weak solution of \eqref{eq6}.
 \end{proof}

Next we show some convergence properties of the solutions to~\eqref{eq:frac_crit_exp} as $s_k\to s$, where $s>0$. We begin with an auxiliary lemma.

 \begin{lem}\label{lem:unbd}Assume the hypothesis of Theorem \ref{thm:existence_teo}.  Let $s>0$, $N>2s,$ and $0<\delta\leq \min\{s,\frac{N}{2}-s\}$. For each $t\in(s-\delta,s+\delta)$, let $u_t$ be a least-energy solution to~\eqref{eq6} given by Theorem~\ref{thm:existence_teo}. Then there is a constant $C>1$ depending only on $\delta$ and $s$ such that
\begin{align*}
C^{-1}<\|u_{t}\|_{t}<C\qquad \text{ for all }t\in(s-\delta,s+\delta).
\end{align*}
\end{lem}
\begin{proof}
Repeat the proof from Lemma \ref{lem:bds} with $\Omega=\R^N$.
\end{proof}

\begin{theo}\label{thm:conv:entire} Assume that $G$ and $\phi$ verify assumptions \textnormal{\textbf{($A_1$)}} and \textnormal{\textbf{($A_2$)}}. Let $N\in\N$, $(s_k)_{k\in\N}\subset (0,\infty)$ such that $s_k\to s=m+\sigma>0$ as $k\to \infty$ with $m\in\N_0$, $\sigma\in[0,1]$, and $N>2\max\{s,s_k\}$ for all $k\in\N$. For $\kappa_{N,s}$ as in \eqref{eq:best_constant} and $\tau>0$ such that
\begin{align}\label{tau}
 \tau<\left(3\kappa_{N,s_k}^2\right)^{-\frac{N}{2s_k}}\qquad \text{ for all }k\in\N.
\end{align}
Let $w_{s_k}\in D^{s_k}(\R^N)^\phi$ be a least-energy solution of $(-\Delta)^{s_k} w_{s_k}=|w_{s_k}|^{2^\star_{s_k}-2}w_{s_k}$ satisfying that
\begin{align}\label{rscl}
\int_{B_1(0)} |w_{s_k}|^{2^\star_{s_k}}=\tau\qquad \text{ for all }k\in\N.
\end{align}
Then, there is a least-energy solution $w\in D^s(\R^N)^\phi$ of $(-\Delta)^{s} w=|w|^{2^\star_{s}-2}w$ such that, up to a subsequence, 
\begin{align}\label{c1}
\text{$\eta w_{s_k}\to \eta w$ in $D^{s-\delta}(\R^N)$  as $k\to\infty$ for all $\eta\in C^\infty_c(\R^N)$ and $\delta\in(0,\sigma)$.} 
\end{align}
\end{theo}
\begin{proof}
 Let $s_k$ and $w_{s_k}$ as in the statement.  In the following, $C>0$ denotes possibly different constants independent of $k$.  By Lemma~\ref{lem:unbd}, there is $C>0$ such that
\begin{align}\label{ab}
C^{-1}<c_{s_k}^\phi(\R^N)<C\qquad \text{ for all }k\in\N.
\end{align}
We split the proof in steps. 

\medskip

\emph{Step 1: Find a limit profile for $w_{s_k}$.}  Let $\zeta\in C^\infty_c(\R)$ be such that
\begin{align}\label{zeta}
0 \leq \zeta \leq 1\quad \text{ in }\R,\quad
\zeta(r)=1\quad \text{ if }|r|\leq 1,\quad
\zeta(r)=0\quad \text{ if }|r|\geq 2,
\end{align}
and let
\begin{align}\label{zetan}
w_{s_k}^n(x):=w_{s_k}(x)\zeta_n(x),\qquad \zeta_n(x):=\zeta \left( \frac{|x| }{n} \right)\qquad \text{ for }n\in\N\text{ and }x\in \R^N. 
\end{align}
By Lemma \ref{GV} and triangle inequality, 
\begin{align}\label{wsknbd}
\|w_{s_k}^n\|_{s_k}<C\quad \text{ for all } n,k\in\N.
\end{align}
By Lemma~\ref{aship:lem}, there is $w^n_s\in D^s(\R^N)^\phi$ such that, up to a subsequence,
\begin{align}\label{wk}
 \varphi w^n_{s_k}\to \varphi w^n_s\quad \text{ in }D^{s-\delta}(\R^N)\text{ as }k\to\infty\text{ for all }n\in\N\text{ and }\varphi\in C^\infty_c(\R^N),
\end{align}
and, by a standard diagonalization argument, we may assume that $w^n_s = w^m_s$ in $B_n(0)$ for all $m>n,$ $m,n\in\N.$  Moreover, using Fatou's Lemma (as in \eqref{fatou}) and \eqref{wsknbd}, $\|w_s^n\|_{s}\leq \liminf\limits_{k\to\infty}\|w_{s_k}^n\|_{s_k}<C$ for all $n\in\N,$ and therefore there is $w\in D^s(\R^N)^\phi$ such that, up to a subsequence,
\begin{align}\label{wk2}
 w^n_s \weakly w\quad \text{ weakly in }D^s(\R^N)\quad \text{ as }n\to\infty,
\end{align}
but then $w^n_s=w$ in $B_n(0)$ for all $n\in \N$ and, by Lemma \ref{aship:lem}, we deduce that
\begin{align}\label{app2}
 w_{s_k}\to w\quad \text{ in }L^q_{loc}(\R^N)\quad \text{ as }k\to\infty\quad \text{ for }q\in[1,2^\star_s)
\end{align}
and \eqref{c1} follows from \eqref{wk} taking $n$ large enough.

\medskip

\emph{Step 2: Show that $w$ is a weak solution.}  Let $\varphi\in C^\infty_c(\R^N)$ and $n\in\N$. Observe that, by  \eqref{app2} and Lemma \ref{phi:con},
\begin{align}
\lim_{k\to\infty}\int_{B_n(0)}w_{s_k}(-\Delta)^{s_k}\varphi
=\int_{B_n(0)}w(-\Delta)^{s}\varphi
=\int_{\R^N}w(-\Delta)^{s}\varphi
\label{id1}
\end{align}
and, by Hölder's inequality, \eqref{wsknbd}, Lemma \ref{thm:sobolev}, and Lemma \ref{phi:con},
\begin{align}
\lim_{k\to\infty}&\int_{\R^N\backslash B_n(0)}w_{s_k}(-\Delta)^{s_k}\varphi
\leq \lim_{k\to\infty}|w_{s_k}|_{2^\star_{s_k}}
\left(\int_{\R^N\backslash B_n(0)}|(-\Delta)^{s_k}\varphi|^{(2^\star_{s_k})'}\right)^\frac{1}{(2^\star_{s_k})'}\notag\\
&\leq C\lim_{k\to\infty}\left(\int_{\R^N\backslash B_n(0)}|(-\Delta)^{s_k}\varphi|^{{\frac{2N}{N+2s_k}}}\right)^{\frac{N+2s_k}{2N}}= C\left(\int_{\R^N\backslash B_n(0)}|(-\Delta)^{s}\varphi|^{{\frac{2N}{N+2s}}}\right)^{\frac{N+2s}{2N}}=0,\label{id2}
\end{align}
where $(-\Delta)^{s}\varphi\in L^\frac{2N}{N+2s}(\R^N)$ by Lemma \ref{phi:con2}.  Then, \eqref{id1}, \eqref{id2}, and integration by parts (see \emph{e.g.} \cite[Lemma 1.5]{AJS18}) imply that, 
\begin{align*}
 \lim_{k\to\infty}\cE_{s_k}(w_{s_k},\varphi)
 &=\lim_{n\to\infty}\lim_{k\to\infty}\int_{B_n(0)}w_{s_k}(-\Delta)^{s_k}\varphi+\int_{\R^N\backslash B_n(0)}w_{s_k}(-\Delta)^{s_k}\varphi=\int_{\R^N}w(-\Delta)^{s}\varphi
 =\cE_s(w,\varphi).
\end{align*}
Therefore, by \eqref{app2},
\begin{align}\label{wwsol}
0=\lim_{k\to \infty}J_{s_k}^\prime(w_{s_k})\varphi
=\cE_s(w,\varphi)-\int_{\R^N}|w|^{\dsa{s}-2}w_{s}\varphi=J'_s(w)\varphi
\end{align}
and $w$ is a weak solution of the limit problem. 

\medskip

\emph{Step 3: Verify that}
\begin{align}\label{wne0}
w\neq 0. 
\end{align}
Assume, by contradiction, that $w=0$ and let $\varphi\in C^\infty_c(B_1(0))$; then, since $w_{s_k}$ is a weak solution and $\varphi^2 w_{s_k}\in D^{s_k}(\R^N)$ (by Lemma \ref{lem:prod_co}),
\begin{equation}\label{eq:asymp_func:3}
J_{s_k}^\prime(w_{s_k})(\varphi^2 w_{s_k})=0\quad \text{ for all }k\in\N.
\end{equation}

Then, by \eqref{wk} and Lemma~\ref{lem:bk},  
\begin{align*}
\|w_{s_k}\varphi\|^2_{s_k}\leq \frac{3}{2}\cE_{s_k}(w_{s_k},\varphi^2 w_{s_k})+o(1)\quad \text{ as } k\to\infty.
\end{align*}
Therefore, by H\"older's inequality, \eqref{tau}, \eqref{rscl}, \eqref{eq:asymp_func:3}, and the fact that $\operatorname{supp}(\varphi)\subset B_1(0)$,
\begin{align*}
\norme{w_{s_k}\varphi}_{s_k}^2 &\leq \frac{3}{2}\int_{B_1(0)}|w_{s_k}|^{\dsa{s_k}-2}|w_{s_k}\varphi|^2+o(1)\\
&\leq \frac{3}{2}\left(\int_{B_1(0)}|w_{s_k}|^{\dsa{s_k}}\right)^{\frac{\dsa{s_k}-2}{\dsa{s_k}}}\left(\int_{\R^N}|\varphi w_{s_k}|^{\dsa{s_k}}\right)^{\frac{2}{\dsa{s_k}}}+o(1)\leq \frac{3}{2}\tau^{\frac{2s_k}{N}} |\varphi w_{s_k}|^2_{\dsa{s_k}}+o(1)
\end{align*}
as $k\to\infty$.  Using \Cref{thm:sobolev} and \eqref{tau}, we have that
\begin{align} \label{eq:bound_norm_claim}
\norme{\varphi w_{s_k}}_{s_k}^2 \leq \frac{3}{2}\tau^{\frac{2s_k}{N}}\kappa_{N,s_k}^2 \norme{\varphi w_{s_k}}^2_{s_k}+o(1)
 \leq \frac{1}{2} \norme{\varphi w_{s_k}}^2_{s_k}+o(1).
\end{align}
Then, by~\eqref{eq:bound_norm_claim}, $\|\varphi w_{s_k}\|_{s_k}=o(1)$ and therefore, by Theorem \ref{thm:sobolev}, $|\varphi w_{s_k}|_{2^\star_{s_k}}=o(1)$ as $k\to\infty$ for any $\varphi\in C_c^\infty(B_1(z))$, which contradicts~\eqref{rscl}. Therefore \eqref{wne0} holds.

\medskip 

\emph{Step 4: Show that $w$ is a $\phi$-equivariant least-energy solution.}  Observe that, by \eqref{ab}, there is $c_*$ such that $c_{s_k}^\phi(\R^N)\to c_*$ as $k\to\infty$ up to a subsequence. Moreover, by \eqref{wwsol} and \eqref{wne0}, we have that $w\in\cN_s$.
By \eqref{app2} we have that $|w_k|^{2^\star_{s_k}}\to |w|^{2^\star_{s}}$ pointwisely in $\R^N$ as $k\to\infty$; then, by Fatou's Lemma,
\begin{align*}
 \left(\frac{1}{2}-\frac{1}{2_s^\star}\right)|w|^{2^\star_s}_{2^\star_s}
 \leq \liminf_{k\to\infty}\left(\frac{1}{2}-\frac{1}{2_{s_k}^\star}\right)|w_k|^{2^\star_{s_k}}_{2^\star_{s_k}} =\liminf_{k\to\infty}c_{s_k}^\phi\leq c_*,
\end{align*}
and then, by minimality,
\begin{align}
 &c_s^\phi(\R^N)
 \leq J_s(w) 
 =\left(\frac{1}{2}-\frac{1}{2_s^\star}\right)|w|_{2^\star_{s}}^{2^\star_{s}}
 \leq c_*.\label{3ubd}
\end{align}
On the other hand, by Theorem \ref{thm:existence_teo}, there is a $\phi$-equivariant least-energy solution $u_s\in \cN_s$ such that $J_s(u_s)=c_s^\phi(\R^N)$. By density, there is a sequence $(u_{s,n})_{n\in\N}\subset C^\infty_c(\R^N)$ such that $u_{s,n}\to u_s$ in $D^{s}(\R^N)$ as $n\to\infty$. Let
\begin{equation}\label{1ubd}
t_{k,n}:=\left(\frac{\|u_{s,n}\|_{s_k}^2}{|u_{s,n}|_{2^\star_{s_k}}^{2^\star_{s_k}}}\right)^{\frac{1}{2^\star_{s_k}-2}}.
\end{equation}
Then $\lim\limits_{n\to \infty}\lim\limits_{k\to\infty}t_{k,n}=1$.  Moreover,
\begin{equation}\label{2ubd}
\lim_{n\to \infty}\lim_{k\to\infty}\| u_{s,n}\|_{s_k}=\| u_{s}\|_s, \quad 
\lim_{n\to \infty}\lim_{k\to\infty}|u_{s,n}|_{2^\star_{s_k}}=|u_s|_{2_s^\star}.
\end{equation}
In particular,
\begin{align*}
\lim_{n\to \infty}\lim_{k\to\infty} J_{s_k}(t_{k,n}u_{s,n}) = J_s(u_s)=c_s^\phi.
\end{align*}
But then, using the minimality of $w_{s_k}$, \eqref{1ubd}, and \eqref{2ubd},
\begin{align*}
c_*^\phi+o(1)&=c_{s_k}^\phi=J_{s_k}(w_{s_k})\leq J_{s_k}(t_{k,n}u_{s,n})\qquad\text{ as }k\to\infty,
\end{align*}
and therefore $c_*^\phi\leq \lim\limits_{n\to \infty}\lim\limits_{k\to\infty} J_{s_k}(t_{k,n}u_{s,n}) = J_s(u_s)=c_s^\phi(\R^N)$. Together with \eqref{3ubd}, we conclude that $c_*^\phi(\R^N)=c_s^{\phi}(\R^N)$ and that $J_s(w)=c_s^\phi(\R^N)$. This establishes that $w$ is a $\phi$-equivariant least-energy solution of the limiting problem, and ends the proof.
\end{proof}

\begin{proof}[Proof of Theorem~\ref{main:thm:unbdd}]
 The first part (existence) follows from Theorem~\ref{thm:existence_teo} and the second part (convergence) follows from Theorem~\ref{thm:conv:entire}.  Observe that, due to the scaling invariance \eqref{issues}, an arbitrary function $w\in D^{s_k}_0(\R^N)^\phi$ has a rescaling $\widetilde w\in D^{s_k}_0(\R^N)^\phi$ satisfying \eqref{rscl}.
\end{proof}

\appendix

\section{Auxiliary Lemmas}

\subsection{Convergence of test functions} 

In this subsection we show a uniform bound for $|(-\Delta)^{s_k}\varphi|$ whenever $s_k\to s>0$ and $\varphi\in C^\infty_c(\R^N)$ and show the convergence of $((-\Delta)^{s_k}\varphi)_{k\in\N}$ in $L^p(\R^N)$ for any $p\geq 1$.  For $s>0$ and $k\in \N$ denote 
 \begin{equation}\label{sks-space}
 S^{n}_{t}:=\{\psi\in C^n(\R^N)\;:\; \sup_{x\in \R^N}(1+|x|^{N+2t})\sum_{|\alpha|\leq n}|\partial^{\alpha}\psi(x)| <\infty\}
 \end{equation}
 endowed with the norm $\|\psi\|_{n,t}:= \sup\limits_{x\in \R^N}(1+|x|^{N+2t})\sum\limits_{|\alpha|\leq n}|\partial^{\alpha}\psi(x)|$.
 
 The next Lemma is a version of \cite[Lemma B.5]{AJS18_green} with uniform estimates. 
\begin{lem}\label{phi:con2}
	Let $m\in \N_0$, $(\sigma_k)_{k\in\N}\subset[0,1]$, $\lim_{k\to\infty}\sigma_k=:\sigma\in[0,1]$, $s_k:=m+\sigma_k>0$, $s:=m+\sigma>0$, $\delta\in(0,s)$, and $\varphi\in C^\infty_c(\R^N)$.  There is ${C=C(N,s,\delta)>0}$ such that, passing to a subsequence,
	\begin{align}\label{bd:eq}
	|(-\Delta)^{s_k}\varphi(x)|\leq C\frac{\|\varphi\|_{2m+2,s-\delta}}{1+|x|^{N+2(s-\delta)}}\qquad \text{  for all $x\in \R^N$ and $k\in\N$}.
	\end{align}
\end{lem}
\begin{proof}
It suffices to consider $\sigma_k\in(0,1)$. Note that $(-\Delta)^{m+\sigma_k} \varphi= (-\Delta)^{\sigma_k}(-\Delta)^{m} \varphi$ (this follows by Fourier transform, see also \cite[Theorems 1.2 and 1.9]{AJS18} for a proof via direct calculations). Let $\psi:=(-\Delta)^m \varphi$, $B:=B_1(0)$, and $\delta\in(0,s)$.  In the following $C>0$ denotes possibly different constants depending at most on $N$, $s$, and $\delta.$

By \eqref{cnsigma},
\begin{align}\label{bound_sigma_k}
c_{N,\sigma_k}\leq C\sigma_k(1-\sigma_k)\qquad \text{ for all } k\in\N.
\end{align}
For $x\in\R^N$ we have, by the mean value theorem (see \cite[Lemma B.1]{AJS18_green}),
	\begin{align}\label{A5}
	|&(-\Delta)^{\sigma_k+m} \varphi(x)|=\frac{c_{N,\sigma_k}}{2}\Bigg|\int_{\R^N} \frac{2\psi(x)-\psi(x+y)-\psi(x-y)}{|y|^{N+2\sigma_k}}\d{y}\Bigg|\notag\\
	&\leq C\sigma_k(1-\sigma_k)\left(\Bigg|\int_{B} \frac{2\psi(x)-\psi(x+y)-\psi(x-y)}{|y|^{N+2\sigma_k}}\d{y}\Bigg|+\Bigg|\int_{\R^N\backslash B} \frac{2\psi(x)-\psi(x+y)-\psi(x-y)}{|y|^{N+2\sigma_k}}\d{y}\Bigg|\right)\notag\\
	&\leq C\sigma_k(1-\sigma_k)\left(\int_{B}\int_0^1\int_0^1\frac{|H_{\psi}( x+(t-\tau)y)|}{|y|^{N+2\sigma_k-2}}\d{\tau} \d{t} \d{y}+\Bigg|\int_{\R^N\setminus B} \frac{2\psi(x)-\psi(x+y)-\psi(x-y)}{|y|^{N+2\sigma_k}}\d{y}\Bigg|\right)\notag\\
	&=:f_{1,k}+f_{2,k},
	\end{align}
	where $H_\psi$ denotes the Hessian of $\psi$. Let $t:=s-\delta$ and note that 
	\begin{align}\label{A4}
	f_{1,k}&\leq C\sigma_k(1-\sigma_k)\|\varphi\|_{2m+2,t}\int_{B}\int_0^1\int_0^1\frac{|y|^{-N-2\sigma_k+2}}{1 +|x+(t-\tau)y|^{N+2t}}\d{\tau} \d{t} \d{y}  \leq C\sigma_k\frac{\|\varphi\|_{2m+2,t}}{1+|x|^{N+2t}},\\
	f_{2,k}&\leq 2\sigma_k(1-\sigma_k)\left(\int_{\R^N\setminus B} \frac{|\psi(x)|}{|y|^{N+2\sigma_k}}\d{y}+\Bigg|\int_{\R^N\setminus B} \frac{\psi(x+y)}{|y|^{N+2\sigma_k}}\d{y}\Bigg|\right)\notag\\
	&\leq C\sigma_k(1-\sigma_k)\left(\frac{1}{\sigma_k}\frac{\|\varphi\|_{2m+2,t}}{1+|x|^{N+2t}}+\Bigg|\int_{\R^N\setminus B} \frac{\psi(x+y)}{|y|^{N+2\sigma_k}}\d{y}\Bigg|\right).\label{A3}
	\end{align}
	Using integration by parts $m$ times we obtain
	\begin{align}\label{A2}
	\Bigg |\int_{\R^N\setminus B} \frac{\psi(x+y)}{|y|^{N+2\sigma_k}}\d{y}\Bigg|=\Bigg |\int_{\R^N\setminus B} \frac{(-\Delta)^m \varphi(x+y)}{|y|^{N+2\sigma_k}}\d{y}\Bigg|\leq C\frac{\|\varphi\|_{2m+2,t}}{1+|x|^{N+2t}}+ C\int_{\R^N\setminus B} \frac{|\varphi(x+y)|}{|y|^{N+2\sigma_k+2m}}\d{y}.
	\end{align}
	Moreover,
	\begin{align}\label{A1}
	\int_{\R^N\setminus B} \frac{|\varphi(x+y)|}{|y|^{N+2\sigma_k+2m}}\d{y}\leq \frac{\|\varphi´\|_{2m+2,t}}{1+|x|^{N+2t}}\int_{\R^N\setminus B} \frac{1+|x|^{N+2t}}{(1+|x+y|^{N+2t})|y|^{N+2s_k}}\d{y}.
	\end{align}
	
	By \eqref{A5}-\eqref{A1} it suffices to show that
	\begin{align}\label{A6}
	\int_{\R^N\setminus B} \frac{1+|x|^{N+2t}}{(1+|x+y|^{N+2t})|y|^{N+2s_k}}\d{y}<C\qquad \text{for all $x\in\R^N$.}
	\end{align}
	If $|x|<2$ then \eqref{A6} follows by taking the maximum over $x\in 2B$. Fix $|x|\geq 2$ and let $U:=\{y\in\R^N\backslash B : |x+y|\geq \frac{|x|}{2}\}$. If $y\in U$ then $1+|x|^{N+2t}\leq C(1+|x+y|^{N+2t})$ and if $y\in \R^N\backslash U$ then $|y|>\frac{|x|}{2}$.  Thus,
	\begin{align*}
	&\int_{U} \frac{1+|x|^{N+2t}}{(1+|x+y|^{N+2t})|y|^{N+2s_k}}\d{y} \leq C \int_{\R^N\backslash B} |y|^{-N-2s_k}\d{y}=\frac{C}{s_k}<C,\\
	&\int_{\R^N\backslash U} \frac{1+|x|^{N+2t}}{(1+|x+y|^{N+2t})|y|^{N+2s_k}}\d{y} \leq C\frac{1+|x|^{N+2t}}{|x|^{N+2s_k}}\int_{\R^N}\frac{1}{1+|x+y|^{N+2t}}\d{y}<C.
	\end{align*}
	This implies \eqref{A6} and finishes the proof.
\end{proof}

\begin{lem}\label{phi:con}
Let $\sigma\in[0,1]$, $m\in \N_0$, $(\sigma_k)_{k\in\N}\subset[0,1]$, $\lim_{k\to\infty}\sigma_k=:\sigma\in[0,1]$, $s_k:=m+\sigma_k>0$, $s:=m+\sigma>0$, and $\varphi\in C^\infty_c(\R^N)$. Then $(-\Delta)^{s_k}\varphi \to (-\Delta)^{s}\varphi$ in $L^p(\R^N)$ as $k\to\infty$ for any $p\geq 1$.
\end{lem}
\begin{proof}
 That $(-\Delta)^{s_k}\varphi\to (-\Delta)^{s}\varphi$ pointwisely in $\R^N$ as $k\to\infty$ follows using Fourier transform.  In fact, the function $t\mapsto (-\Delta)^{t}\varphi(x)$ is analytic in $(0,\infty)$ for every $x\in\R^N$, see \cite[Lemma 2]{DG17}.  Let $\delta\in(0,s)$ and $p\geq 1$. By Lemma \ref{phi:con2}, there is $C=C(N,s,\delta,\varphi)>0$ such that, for $k$ large enough,
 $|(-\Delta)^{s_k}\varphi-(-\Delta)^{s}\varphi|^p
  \leq C(1+|x|^{N+2(s-\delta)})^{-p}\in L^1(\R^N)$.  The claim now follows by dominated convergence. 
 \end{proof}

\subsection{Uniform bounds} 

The goal of this subsection is to show two auxiliary results employed in the proof of \Cref{thm:conv:entire}.

 \begin{lem}\label{lem:prod_co}
Let $s>0$, $w\in D^{s}(\R^N)$, and $\eta\in C^\infty_c(\R^N)$, then $w\eta\in D^s(\R^N)$ and $\norme{w\eta}_{s}\leq |\hat \eta|_1^2\|w\|_s.$
\end{lem}
\begin{proof}
Let $\eta\in C^\infty_c(\R^N)$. Then, by the convolution theorem,
\begin{align*}
\norme{w\eta}_{s}^2
 &=\int_{\R^N}|\xi|^{2s}|\widehat {w\eta}|^2\d{\xi}
 =\int_{\R^N}|\xi|^{2s}|\widehat {w}*\widehat{\eta}(\xi)|^2\d{\xi}
 =\int_{\R^N}|\xi|^{2s}
 \left|\int_{\R^N}\widehat {w}(\xi-y)\widehat{\eta}(y)\d{y}\right|^2\d{\xi}\\
 &=\int_{\R^N}|\xi|^{2s}
 \left|\int_{\R^N}\widehat {w}(\xi)e^{i\xi\cdot y}\widehat{\eta}(y)\d{y}\right|^2\d{\xi}
 \leq \int_{\R^N}|\xi|^{2s}|\widehat {w}(\xi)|^2
 \left|\int_{\R^N}|\widehat{\eta}(y)|\d{y}\right|^2\d{\xi}
 =|\hat \eta|_1^2\|w\|^2_{s}.\qedhere
 \end{align*}
\end{proof}

 \begin{lem}\label{GV}
Let $m\in\mathbb N_0$, $(\sigma_k)_{k\in\N}\subset[0,1]$, and $\lim_{k\to\infty}\sigma_k=:\sigma\in[0,1]$.  For $k\in \N$, let $s_k:=m+\sigma_k>0$, $s:=m+\sigma>0$, and $w_{k}\in D^{s_k}(\R^N)$.  If
\begin{align}\label{bd:a}
\|w_k\|_{s_k}<C_1\qquad \text{ for all $k\in\N$ and for some $C_1>0$},
\end{align}
then $\norme{\zeta_nw_k - w_k}_{s_k} \leq C$ for $\zeta_n$ as in \eqref{zetan} and some $C>0$ uniform with respect to $n$ and $k$.

\end{lem}

Lemma \ref{GV} follows essentially from \Cref{lem:est_gen} below. We argue as in \cite[Proposition B.1]{BGV18}, which shows that $\zeta_n w\to w$ as $n\to\infty$ in $D^s(\R^N)$ for $s\in(0,1)$. When considering $w_k\in D^{s_k}(\R^N)$ instead of $w$, showing convergence is much more delicate and requires additional assumptions.  Here, we only show a uniform bound, which suffices for our purposes.

 Let $B_r(0)=B_r$ denote the ball in $\R^N$ of radius $r>0$ centered at zero. Let $\zeta_n$ be as in \eqref{zetan} and $\varphi_n=:1-\zeta_n$. Clearly,
\begin{equation}\label{eq:prop_phi_n}
\varphi_n=0 \mbox{ in $B_n$}, \qquad \varphi_n= 1 \mbox{ in $\R^N\setminus B_{2n}$}, \qquad \sup_{x,y\in\R^N}\frac{|\partial^\beta\varphi_n(x)-\partial^\beta\varphi_n(y)|}{|x-y|} \leq \frac{C}{n^{|\beta|+1}},
\end{equation}
for some constant $C>0$ depending on $\zeta$ (given in \eqref{zeta}) and a multi-index $\beta\in\N_0^N$.

\begin{lem}\label{lem:est_gen}
Let $m\in\mathbb N_0$, $\sigma_k\subset[0,1]$, $\lim_{k\to\infty}\sigma_k=:\sigma\in[0,1]$, and consider multi-indices $\alpha,\beta\in\mathbb N_0^N$ such that $|\alpha|+|\beta|=m$. For $k\in\N$, let $s_k:= m+\sigma_k>0$, $s:=m+\sigma>0,$ and $w_{k}\in D^{s_k}(\R^N)$ such that \eqref{bd:a} holds, then there exists $C>0$ depending at most on $\zeta$ and $N$ such that
\begin{equation}\label{eq:bound_gen_lem}
\norme{\partial^\alpha w_k \partial^\beta \varphi_n }_{\sigma_k} \leq C\qquad \text{ for all }n,k\in\N.
\end{equation}
\end{lem}

\begin{proof} The local case $s_k\in\N$ is clear, so we may assume that {$\sigma_k\in (0,1)$} for all $k\in\mathbb N$. We begin by splitting in suitable subdomains, more precisely,
\begin{align*}
\norme{\partial^\alpha w_{k} \partial^\beta \varphi_n}_\sigma
&= \frac{c_{N,\sigma_k}}{2}\int_{\R^N}\int_{\R^N}  \frac{|\partial^\alpha w_{k}(x) \partial^\beta\varphi_n(x)-\partial^\alpha w_{k}(y) \partial^\beta\varphi_n(y)|^2}{|x-y|^{N+2\sigma_k }} \d{x}\d{y} \\
&= {c_{N,\sigma_k}} \int_{B_{2n}\setminus B_n}\int_{B_n} \cdots
+ {c_{N,\sigma_k}} \int_{\R^N\setminus B_{2n}} \int_{B_n}\cdots
+ \frac{c_{N,\sigma_k}}{2} \int_{B_{2n}\setminus B_n} \int_{B_{2n}\setminus B_n} \cdots\\
&\quad+ {c_{N,\sigma_k}} \int_{\R^N\setminus B_{2n}} \int_{B_{2n}\setminus B_n}  \cdots
+ \frac{c_{N,\sigma_k}}{2} \int_{\R^N\setminus B_{2n}} \int_{\R^N\setminus B_{2n}}\cdots
+ \frac{c_{N,\sigma_k}}{2} \int_{B_{n}} \int_{B_n}\cdots
=: \sum_{i=1}^{6} J_i.
\end{align*}

We show that each $J_i$ is uniformly bounded in $k$ and $n$ for $1\leq i\leq 6$.  
In the following, $C>0$ denotes possibly different constants depending at most on $N$, $m$, $\sigma$, and $\zeta$ (given in \eqref{zeta}).

\medskip

\emph{-Estimate for $J_1$}. By \eqref{eq:prop_phi_n},
\begin{align*}
J_1
&= {c_{N,\sigma_k}} \int_{B_{2n}\setminus B_n}\int_{B_n} \frac{|\partial^\alpha w_{k}(y) \partial^\beta \varphi_n(y)|^2}{|x-y|^{N+2\sigma_k }} \d{x}\d{y}\\
&=  {c_{N,\sigma_k}} \int_{B_{2n}\setminus B_n}\int_{B_n} \frac{|\partial^\beta\varphi_n(x)-\partial^\beta\varphi_n(y)|^2}{|x-y|^{N+2\sigma_k }} |\partial^\alpha w_{k}(y) |^2 \d{x}\d{y}\\
&\leq \frac{Cc_{N,\sigma_k}}{n^{2(|\beta|+1)}}  \int_{B_{2n}\setminus B_n}\int_{B_n} \frac{|\partial^\alpha w_{k}(y)|^2}{|x-y|^{N-2(1-\sigma_k)}}\d{x}\d{y}.
\end{align*}
Note that $B_n\subset B_{3n}(y)$ for $y\in B_{2n}\setminus B_n$ and hence
\begin{align*}
\int_{B_n} \frac{1}{|x-y|^{N-2(1-\sigma_k)}}\dx \leq \int_{B_{3n}(y)}  \frac{1}{|x-y|^{N-2(1-\sigma_k)}}\dx \leq \frac{C}{1-\sigma_k} n^{2(1-\sigma_k)}.
\end{align*}
Thus, 
\begin{equation}\label{J11}
J_{1}\leq \frac{c_{N,\sigma_k}}{1-\sigma_k} \frac{C}{n^{2(|\beta|+\sigma_k)}}\int_{B_{2n}\setminus B_n} |\partial^\alpha w_{k}(y)|^2 \d{y}.
\end{equation}

By Fourier transform and the assumption \eqref{bd:a},
\begin{align}\label{eq:four_equiv}
 \|\partial^\alpha w_{k}\|^2_{s_k-|\alpha|}
 =\int_{\R^N} |\xi|^{2(s_k-|\alpha|)}|\widehat{\partial^\alpha w_k}|^2\d{\xi}
 = \int_{\R^N} |\xi|^{2s_k}|\widehat w_k|^2\d{\xi} =\|w_{k}\|^2_{s_k}<C\quad \text{ for all }k\in\N,
\end{align}
Then, by \Cref{thm:sobolev}, $|\partial^\alpha w_{k}(y)|_{\dsa{s_k-|\alpha|}}<C$ for all $k\in\N$ and, using H\"older inequality,  \eqref{cnsigma}, and \eqref{eq:four_equiv},
\begin{align*}
J_1 &\leq \frac{C}{n^{2(|\beta|+\sigma_k)}} \left(\int_{B_{2n}\setminus B_n} |\partial^\alpha w_{k}(y)|^{\dsa{s_k-|\alpha|}}\right)^{\frac{2}{\dsa{s_k-\alpha}}} \left(Cn^{N}\right)^{\frac{2(s_k-|\alpha|)}{N}}
\leq {C} |\partial^\alpha w_{k}(y)|_{\dsa{s_k-|\alpha|}}^{2}<C
\end{align*}
for all $n,k\in\N$, where we have used that $|\alpha|+|\beta|=m$ and $s_k=m+\sigma_k$. 

\smallskip
\textit{-Estimate for $J_2$.} If $|\beta|>0$, then $J_2\equiv 0$ by \eqref{eq:prop_phi_n}. Thus, let $|\beta|=0$. Then $|\alpha|=m$ and, by \eqref{eq:prop_phi_n},
\begin{align*}
J_2=c_{N,\sigma_k}\int_{\R^{N}\setminus B_{2n}}\int_{B_n}|\partial^\alpha w_{k}(y)|^2 |x-y|^{-N-2\sigma_k}\d{x}\d{y}. 
\end{align*}
By Fubini's theorem and H\"older's inequality,
\begin{equation}\label{eq:J2_fub_hol}
J_2\leq \int_{B_n} \left(\int_{\R^N\setminus B_{2n}}|\partial^\alpha w_{k}(y)|^{\dsa{\sigma_k}}\d{y}\right)^{\frac{2}{\dsa{\sigma_k}}} \left(c_{N,\sigma_k}\int_{\R^N\setminus B_{2n}} |x-y|^{-(N+2\sigma_k)\frac{N}{2\sigma_k}}\d{y}\right)^{\frac{2\sigma_k}{N}} \d{x}.
\end{equation}
Since $B_n(x)\subset B_{2n}$ for every $x\in B_n$ and $s_k-|\alpha|=\sigma_k$, then
\begin{equation*}
c_{N,\sigma_k}\int_{\R^N\setminus B_{2n}} |x-y|^{-(N+2\sigma_k)\frac{N}{2\sigma_k }} \d{y}
\leq c_{N,\sigma_k}\int_{\R^N\setminus B_{n}(x)} |x-y|^{-(N+2\sigma_k)\frac{N}{2\sigma_k }} \d{y} \leq C n^{-\frac{N^2}{2\sigma_k}}.
\end{equation*}
By putting this estimate in \eqref{eq:J2_fub_hol}, we obtain, as before, that, for all $n,k\in\N$,
\begin{equation*}
J_2\leq C n^{-N} \left(\int_{\R^N\setminus B_{2n}}|\partial^\alpha w_{k}(y)|^{\dsa{\sigma_k}}\d{y}\right)^{\frac{2}{\dsa{\sigma_k}}} \left(\int_{B_n}1 \d{x}\right) \leq  C \left(\int_{\R^N}|\partial^\alpha w_{k}(y)|^{\dsa{\sigma_k}}\d{y}\right)^{\frac{2}{\dsa{\sigma_k}}}<C. 
\end{equation*}

\smallskip
\textit{-Estimate for $J_3$.} A straightforward computation yields that
\begin{align*}
J_3&\leq c_{N,\sigma_k} \int_{B_{2n}\setminus B_n} \int_{B_{2n}\setminus B_n} \frac{|\partial^\beta(x) \varphi_n(x)-\partial^\beta \varphi_n(y)|^2}{|x-y|^{N+2\sigma_k}} |\partial^\alpha w_{k}(x)|^2\d{x}\d{y} \\
&\quad + c_{N,\sigma_k}  \int_{B_{2n}\setminus B_n} \int_{B_{2n}\setminus B_n} \frac{|\partial^\alpha w_{k}(x)-\partial^\alpha w_{k}(y)|^2}{|x-y|^{N+2\sigma_k}} | \partial^\beta \varphi_n(y)|^2\d{x}\d{y}.
\end{align*}
By using the mean value theorem in the first integral and \eqref{eq:prop_phi_n} in the second one, we have that
\begin{align*}
J_3&\leq \frac{c_{N,\sigma_k}}{n^{2(|\beta|+1)}} \int_{B_{2n}\setminus B_n} \int_{B_{2n}\setminus B_n} \frac{1}{|x-y|^{N-2(1-\sigma_k)}} |\partial^\alpha w_{k}(x)|^2\d{x}\d{y} \\
&\quad + \frac{c_{N,\sigma_k}}{n^{2|\beta|}}  \int_{B_{2n}\setminus B_n} \int_{B_{2n}\setminus B_n} \frac{|\partial^\alpha w_{k}(x)-\partial^\alpha w_{k}(y)|^2}{|x-y|^{N+2\sigma_k}} \d{x}\d{y}=: J_3^{(1)}+J_{3}^{(2)}.
\end{align*}
An argument similar to the estimate of $J_1$ yields that $J_3^{(1)}< C$ uniformly in $k$ and $n$, whereas $J_3^{(2)}\leq C\|\partial^\alpha w_k\|_{\sigma_k}<C$ for all $k,n\in\N$, by \eqref{eq:four_equiv}.

\smallskip
\textit{-Estimate for $J_4$.} If $|\beta|=0$, then $|\alpha|=m$ and, by \eqref{eq:prop_phi_n},
\begin{align*}
J_4 & \leq c_{N,\sigma_k}\int_{\R^N\setminus B_{2n}} \int_{B_{2n}\setminus B_n} \frac{|\partial^\alpha w_{k}(x)-\partial^\alpha w_{k}(y)|^2}{|x-y|^{N+2\sigma_k}}\d{x}\d{y}\\
&\quad + c_{N,\sigma_k} \int_{\R^N\setminus B_{2n}}\int_{B_{2n}\setminus B_n} \frac{|\varphi_n(x)-\varphi_n(y)|^2}{|x-y|^{N+2\sigma_k}}|\partial^{\alpha} w_k(x)|^2\d{x}\d{y}=: J_4^{(1)}+J^{(2)}_4.
\end{align*}
Note that $J_4^{(1)}\leq C\|\partial^\alpha w_k\|_{\sigma_k}<C$ for all $k,n\in\N$ as in the previous step. On the other hand, by the mean value theorem,
\begin{align*}
J_4^{(2)}&
\leq  {C c_{N,\sigma_k}} \int_{\R^N\setminus B_{3n}}\int_{B_{2n}\setminus B_n}\frac{ |\partial^\alpha w_{k}(x)|^2}{|x-y|^{N+2\sigma_k}}\d{x}\d{y}
 + \frac{C c_{N,\sigma_k}}{n^{2}}\int_{B_{3n}\setminus B_{2n}}\int_{B_{2n}\setminus B_n} \frac{|\partial^\alpha w_{k}(x)|^2}{|x-y|^{N-2(1-\sigma_k)}}\d{x}\d{y}
\end{align*}
and the uniform bounds follow as in the estimates for $J_2$ and $J_1$, respectively.

If $|\beta|>0$, we see that
\begin{equation}
J_4=c_{N,\sigma_k}\int_{\R^N\setminus B_{2n}}\int_{B_{2n}\setminus B_n}\frac{|\partial^\beta \varphi_n(x)-\partial^\beta \varphi_n(y)|^2}{|x-y|^{N+2\sigma_k}} |\partial^\alpha w_k(x)|^2\d{x}\d{y},
\end{equation}
and the uniform bound follows as in the estimate for $J_2$.

\smallskip
\textit{-Estimate for $J_5$.} If $|\beta|>0$, then $J_5=0$. If $|\beta|>0$, then, by \eqref{eq:four_equiv},  for all $n,k\in\N$,
\[
J_5=\frac{c_{N,\sigma_k}}{2}\int_{\R^N\setminus B_{2n}}\int_{\R^N\setminus B_{2n}} \frac{|\partial^\alpha w_k(x)-\partial^\alpha w_{k}(y)|^2}{|x-y|^{N+2\sigma_k}}\d{x}\d{y}\leq \|\partial^\alpha w_k\|_{\sigma_k}<C.
\]

\smallskip
The proof is then finished since $J_6\equiv 0$ by \eqref{eq:prop_phi_n}.
\end{proof}

\begin{proof} [Proof of Lemma \ref{GV}]
We show the case $m\in \N$ even and $(\sigma_k)_{k\in\N}\in(0,1)$. The other cases follow similarly.  Since $m$ is even, 
\begin{equation}
(-\Delta)^{{\frac{m}{2}}}(uv)= (-\Delta)^{\frac{m}{2}} u\,v +\sum_{\substack{\{ \alpha,\beta: |\alpha|+|\beta|={m}, \\ |\alpha|<{m} \} }} \mu_{\alpha,\beta} \partial^{\alpha} u \, \partial^\beta v,\quad \text{for $u,v\in H^m(\R^N)$}
\end{equation}
and for suitable coefficients $\mu_{\alpha,\beta}\in\mathbb N_0$. Then, for $s_k=m+\sigma_k$ and $w_{k}\in D^{s_k}(\R^N)$,
\begin{equation}\label{eq:first_split}
\norme{w_{k}\varphi_n}_{s_k}= \norme{(-\Delta)^{\frac{m}{2}}(w_{k}\varphi_n)}_{\sigma_k} \leq \norme{(-\Delta)^{\frac{m}{2}} w_{k} \varphi_n}_{\sigma_k} + C \sum_{\substack{\{ \alpha,\beta: |\alpha|+|\beta|={m}, \\ |\alpha|<{m} \} }}  \norme{ \partial^{\alpha} w_{k} \, \partial^\beta \varphi_n}_{\sigma_k},
\end{equation}
for some $C=C(m,N)>0$, where $\varphi_n$ is given in \eqref{eq:prop_phi_n}.  The claim now follows from \eqref{eq:first_split} and \Cref{lem:est_gen}.
\end{proof}

\renewcommand{\abstractname}{Acknowledgements}
\begin{abstract}
\end{abstract}
\vspace{-0.5cm}

V. Hernández-Santamaría is supported by the program ``Estancias posdoctorales por M\'exico'' of CONACyT, Mexico.  A. Saldaña is supported by UNAM-DGAPA-PAPIIT grant IA101721, Mexico.  We also thank Prof. Mónica Clapp (IMUNAM) for fruitful discussions. 

\bibliographystyle{plain}

\end{document}